\documentclass[14pt, a4paper]{article}
 \usepackage{amsmath,amssymb}
\usepackage[margin=21mm]{geometry}
\usepackage[utf8]{inputenc}
\usepackage{graphics}
\usepackage{xspace}
\usepackage{color}
 \usepackage{amsthm}
 \usepackage[old]{old-arrows}
\usepackage{amsfonts}
\usepackage{booktabs}
\newcommand{\R}{{\mathbb R}}

\newcommand{\dyle}{\displaystyle}
\newcommand{\dint}{\dyle\int}

\newtheorem{Theorem}{Theorem}[section]
\newtheorem{Corollary}[Theorem]{Corollary}

\newtheorem{Proposition}[Theorem]{Proposition}
\newtheorem{Definition}[Theorem]{Definition}

\newtheorem{remark}[Theorem]{Remark}
\usepackage[mathscr]{euscript}
\usepackage{mathrsfs}
\usepackage{enumerate}
 \usepackage{cite}
\usepackage{hyperref}
\newtheorem{theorem}{Theorem}[section]
\newtheorem{lemma}[theorem]{Lemma}

\newtheorem{definition}[theorem]{Definition}

\catcode`@=12

\makeatother
\usepackage[english]{babel}

 \usepackage[hyperpageref]{backref}
\title{On elliptic problems with mixed operators and Dirichlet-Neumann boundary conditions }
\date{}
\author{ Tuhina Mukherjee$^{1}$\thanks{Corresponding author(tuhina@iitj.ac.in)}, Lovelesh Sharma$^{1}$  \\
     \small $^{1}$ Department of Mathematics, Indian Institute of Technology Jodhpur, Rajasthan 342030, India \\
         }

\providecommand{\keywords}[1]
{
  \small	
  \textbf{\textit{Keywords---}} #1
}
\numberwithin{equation}{section}
\begin{document}
\maketitle \vspace{-1.8\baselineskip}
\begin{abstract}
 In this paper, we study  the existence, nonexistence and multiplicity of positive solutions to the problem given by
 \begin{equation*} \label{1}
    \left\{\begin{split} \mathcal{L}u\: &= \lambda u^{q} + u^{p}, \quad u>0 ~~ \text{in} ~\Omega, \\
      u&=0~~\text{in} ~~{D^c},\\
 \mathcal{N}_s(u)&=0 ~~\text{in} ~~{\Pi_2}, \\
 \frac{\partial u}{\partial \nu}&=0 ~~\text{in}~~ \partial \Omega \cap \overline{\Pi_2}.
    \end{split} \right.\tag{$P_\lambda$}
\end{equation*}
 {where  $D= \left(\Omega \cup {\Pi_2} \cup (\partial\Omega\cap\overline{\Pi_2})\right)$ and $D^c$ is the complement of $D$, $\Omega \subseteq \mathbb{R}^n$ is a non empty open set, $\Pi_{1}$, $\Pi_{2}$ are open subsets of $\mathbb{R}^n\setminus{\bar \Omega }$ such that $\overline{{\Pi_{1}} \cup {\Pi_{{2}}}}= \mathbb{R}^n\setminus{\Omega}$, $\Pi_{1} \cap \Pi_{{2}}=  \emptyset$ and $\Omega\cup \Pi_2$ is a bounded set with smooth boundary}, $\lambda >0$ is a  real parameter, $ 0 < q < 1<p $, $n>2$ and
 $\mathcal{L}=  -\Delta+(-\Delta)^{s},~ \text{for}~s \in  (0, 1).$
 We first present a functional setting to study any problem involving $\mathcal L$ under mixed boundary conditions in the presence of concave-convex power nonlinearity, {for a suitable range of $\lambda$, $q$ and $p$}. Our article also contains results related to Picone's identity, strong maximum principles and comparison principles.
 \end{abstract}
 
 \\
 
\section{Introduction}
We investigate the existence and  qualitative properties of weak solutions to the following problem
\begin{equation*} 
    \left\{\begin{split} \mathcal{L}u\: &= \lambda u^{q} + u^{p}, \quad u>0 ~~ \text{in} ~\Omega, \\
      u&=0~~\text{in} ~~{D^c},\\
 \mathcal{N}_s(u)&=0 ~~\text{in} ~~{\Pi_2}, \\
 \frac{\partial u}{\partial \nu}&=0 ~~\text{in}~~ \partial \Omega \cap \overline{\Pi_2}.
    \end{split} \right.\tag{$P_\lambda$}
\end{equation*}

{ where  $D= \left(\Omega \cup {\Pi_2} \cup (\partial\Omega\cap\overline{\Pi_2})\right)$, $\Omega \subseteq \mathbb{R}^n$ is a non empty open set, $\Pi_{1}$, $\Pi_{2}$ are open subsets of $\mathbb{R}^n\setminus{\bar \Omega }$ such that $\overline{{\Pi_{1}} \cup {\Pi_{{2}}}}= \mathbb{R}^n\setminus{\Omega}$, $\Pi_{1} \cap \Pi_{{2}}=  \emptyset$ and $\Omega\cup \Pi_2$ is a bounded set with smooth boundary}, $\lambda>0$ is a real parameter, $0 < q < 1<p $, $n>2$ and
 \begin{equation}\label{A}
\mathcal{L}=  -\Delta+(-\Delta)^{s},~ \text{for}~s \in  (0, 1).
 \end{equation}
The term "mixed" describes the operator as a combination of local and nonlocal operators. In our case, the operator $\mathcal{L}$ in \eqref{1} is generated by superposition of the classical Laplace operator $-\Delta$ and fractional Laplace operator $(-\Delta)^{s}$, which is for a fixed parameter $s \in (0,1)$  defined by,
$${(- \Delta)^{s}u(x)} = C_{n,s}~ P.V. \int_{\mathbb{R}^n} {\dfrac{u(x)-u(y)}{|x-y|^{n+2s}}} ~ dy. $$ 
The term "P.V." stands for Cauchy's principal value, while $C_{n,s}$ is a normalizing constant whose explicit expression is given by
 $$C_{n,s}= \left( \int_{\mathbb{R}^{n}} {\frac{1-cos({\zeta}_{1})}{|\zeta|^{n+2s}}}~d \zeta \right)^{-1}.$$ 
{ In the literature, there are numerous definitions of nonlocal normal derivatives. We consider the one suggested in \cite{MR3651008} for smooth functions $u$ as
\begin{equation}\label{normal}
\mathcal{N}_{s}u(x)= C_{n,s}\dint_{\Omega} \dfrac{u(x)-u(y)}{|x-y|^{n+2s}}\,dy, \qquad  x\in \mathbb{R}^n\setminus\bar{\Omega}.
\end{equation}
}
 Without discussing the specific instances of how this nonlocal operator appears in real-world situations and the reasons behind studying problems involving such operators, we suggest referring to the well-known Hitchhiker's Guide \cite{MR2944369} and the accompanying references for a deeper understanding of the topic.
 The numerous applications serve as the motivation for the study of mixed operators of the kind $\mathcal{L}$ in the problem \eqref{1} such as the theory of optimum searching, biomathematics, and animal forging, we refer \cite{dipierro2022non, MR4249816, MR2924452, MR3590678}.  {Other} common uses include heat transmission in magnetized plasmas, see \cite{blazevski2013local}. These types of operators naturally develop{;} in the applied sciences to investigate the
changes in physical phenomena that have both local and nonlocal effects and these operators occur, for example, in bi-modal power law distribution systems, see \cite{MR4225516} and also this type of operators emerge in models derived from combining two distinct scaled stochastic processes, for a comprehensive explanation of this phenomenon, we refer \cite{MR1640881}. In recent years, there has been a considerable focus on the investigation of elliptic problems involving mixed-type operator $\mathcal{L}$, as in \eqref{1}, exhibiting both local and nonlocal behaviour. 
Let us put some light on literature involving problems with mixed operator $\mathcal L$, among the existing long list. {Cassani et al. in \cite{Cassani} have studied spectral properties, existence, nonexistence and regularity results for the solutions to the class of one-parameter family of elliptic equations which involves mixed local and nonlocal operators}. Biagi, Dipierro, Valdinoci {and  Vecchi} \cite{MR4387204} have provided a comprehensive analysis of a mixed local and nonlocal elliptic problem. In their study, they establish the existence of solutions to the problem, explore maximum principles that govern the behaviour of these solutions, and investigate the interior Sobolev regularity of the solutions. Their findings contribute to a deeper understanding of the properties and behavior of solutions to the following type of problem
\begin{equation}\label{B} 
      \begin{cases}    
     \mathcal{L}u =  g(u), \quad u>0 \quad 
 \text{in }\Omega,\\ 
 ~~u=0 ~~\text{in} ~~\mathbb{R}^n\setminus \Omega. \\
 \end{cases} 
 \end{equation}
 They have studied various interesting inequalities pertaining to mixed operator $\mathcal L$ in \cite{ MR4391102}. Lamao et al. in \cite{MR4357939} investigate the behaviour and properties of solutions to problem \eqref{B}, specifically focusing on their summability characteristics. Arora et. al \cite{arora2021combined} studied the existence and non-existence results of the following type of (singular as well as non-singular) problem 
\[\mathcal L u = \frac{f(x)}{u^\gamma},\; u>0\;\text{in}\; \Omega,\; u=0\;\text{in}\; \mathbb R^n \setminus \Omega,\]
where $n\geq 2$, $\gamma\geq 0$ and $f$ are in a suitable class of Lebesgue functions.  We also refer \cite{MR3445279} for interested readers, where we point out that \cite{MR3445279} contains a study of the Neumann problem with mixed operator. The non linear generalisation of $\mathcal L$ given by $-\Delta_p +(-\Delta_p)^s$ has also started gaining attention, relating them we cite \cite{MR4444761,MR4142782}. 
Dipierro et. al in \cite{MR4438596} was the first who consider mixed operator problems in the presence of classical as well as non-local  Neumann boundary conditions. Their recent article  discusses the spectral properties and the $L^\infty$
bounds associated with a mixed local and nonlocal problem, also in relation to some concrete motivations arising from
population dynamics and mathematical biology. It is worth commenting that none of these articles studies mixed operator problems with mixed Dirichlet- Neumann boundary conditions. This made us inquisitive about what happens when we set up a PDE involving mixed operator $\mathcal L$ under boundary conditions involving Dirichlet datum in some part and Neumann(local and non-local both) datum in the remaining part? { Our paper is distinguished by its approach to answering this question, which involved concave-convex type nonlinearity(see \eqref{1}) and the combination of mixed operator as well as mixed Dirichlet- Neumann boundary, which is the striking feature of our paper.} There exists no single article at present that studies such a problem with any kind of nonlinearity, hence our article is a breakthrough in this regard. We have also established the well known Picone's identity, comparison and strong maximum principles.
Last but not least, we try to provide a glimpse of problems available in literature, involving Dirichlet-Neumann's mixed boundary datum to the readers. Over the past few decades, mixed Dirichlet and Neumann boundary problems associated with elliptic equations 
  \begin{equation*} { -\Delta u = g(x,u)}  ~~\;\text{in}\;~~ \Omega,
 \end{equation*}
 has been extensively studied. In particular, the study of wave phenomena, heat transfer, electrostatics, elasticity and hydrodynamics in physics and engineering often involves elliptic problems with mixed Dirichlet-Neumann boundary conditions. 
  For example, one can see \cite{MR647124,MR1472351,MR2378088} and the references therein for motivation. Very recently, researchers have started looking at Dirichlet-Neumann mixed boundary problems associated with fractional Laplace operator, see \cite{MR3912863, MR2078192, MR4319042, MR4343111}.
  Focusing on the following nonlocal problem with mixed boundary and concave-convex type nonlinearity
  \begin{equation}\label{C} 
      \begin{cases}
      (-\Delta)^{s}u =  \lambda u^{q} + u^{p}, \quad u>0 \quad
 \text{in }\Omega,\\
 ~~~~\mathcal A(u) = 0 \quad \text{on}~ \partial\Omega =  \sum_{D} \cup \sum_{\mathcal{N}},
  \end{cases}
 \end{equation}
 where  $\Omega {\subset} \mathbb{R}^N$  {is} a smooth bounded domain, $0<s<1$, $\lambda >0$ is a  real parameter, {$N>2s$}  and
 $$\mathcal{A}(u)= u \mathbf{1}_{\sum_{D}} + {\partial_{\nu}u}\mathbf{1}_{ \sum_{\mathcal{N}}}, \quad{\partial_{\nu}=\frac{\partial }{\partial{\nu}}},~\text{and}~ \mathbf{1}~\text{denotes characteristic function},$$
we state that  the case $ 0 < q < 1<p<\frac{N+2s}{N-2s} $ has been studied in \cite{MR4151104} whereas the critical exponent case $p=\frac{N+2s}{N-2s} $ has been recently investigated in \cite{MR4535301}. Here $\sum_{D}$, $\sum_{\mathcal{N}}$ are smooth (N-1)-dimensional submanifolds of $\partial \Omega$ such that  $\sum_{D} \cap \sum_{\mathcal{N}}=  \emptyset $. Bisci et al. in \cite{Molica} have studied subcritical nonlocal problems with mixed boundary conditions.  Here, the authors studied the existence of multiple solutions  by using variational and topological arguments based on linking and
$\nabla$-theorems. 
Abdellaoui et al. in \cite{MR3809115} proved the existence of monotonic minimal solutions for \eqref{C} and multiple solutions in the subcritical case. We prove similar results as in \cite{MR3809115} pertaining to problem \eqref{1}, in the same spirit. 
Due to the nature of the problem \eqref{1}, it can be approached using the variational methods. Precisely, the desired solution can be regarded as critical points of the  $C^1$ functional $J_\lambda: {\mathcal{X}^{1,2}}(D)\to \mathbb R$ defined by
\begin{equation}\label{ee1.1}
    J_{\lambda}(u):= \frac{1}{2}\int_{\Omega} |\nabla u|^2~ dx+ \frac{C_{n,s}}{4}\int_{{D_{\Omega}}} \frac{|u(x)-u(y)|^2}{|x-y|^{n+2s}} \\ dxdy -\frac{\lambda}{q+1}\|u^+\|^{q+1}_{L^{q+1}(\Omega)}-\frac{1}{p+1}\|u^+\|^{p+1}_{L^{p+1}(\Omega)},\\
\end{equation}
 where {${D_\Omega =\mathbb{R}^{2n} \backslash (\Omega^{c} \times \Omega^{c})}$,
 $u^{+}= \max\{u,0\}$ and $\mathcal{X}^{1,2}(D)$ is defined in the next section.}

Following are the main results concerning \eqref{1}.
\begin{theorem}\label{T1.1}
Suppose $0<s<1$ and $0<q<1<p$, then there exists a $\Lambda>0$ such that
\begin{enumerate}
    \item Problem \eqref{1} has a minimal solution $u_{\lambda}$ such that $J(u_{\lambda})<0$ when $\lambda \in (0,\Lambda)$, satisfying $u_{\lambda_1}<u_{\lambda_2}$ for $\lambda_1<\lambda_2.$
    \item Problem \eqref{1} has no weak solution, if $\lambda>\Lambda$. 
    \item Problem \eqref{1} has at least one weak solution, if $\lambda=\Lambda.$
\end{enumerate}
\end{theorem}
\begin{theorem}\label{T2.1}
    Suppose $0<s<1$, $0<q<1<p< \frac{n+2}{n-2}$ and $\lambda\in (0,\Lambda)$, then problem \eqref{1} has a second solution $v_\lambda$ such that $v_\lambda > u_\lambda$.
\end{theorem}
\numberwithin{equation}{section}
The structure of our article is as follows: Section 2 provides an introduction to the functional framework required to address the problem \eqref{1}. It presents the specific notion of the solution that will be employed and introduces relevant auxiliary results.
{Section 3 focuses on establishing the existence of solutions, strong maximum principles for an auxiliary problem, which are essential for proving our main Theorems. In section 4, we focus on the existence of minimal and extremal solutions for problem \eqref{1} and demonstrate the existence of a second solution by employing Alama's argument. This section highlights the specific approach and methodology used to obtain our main results.

\section{Function spaces}
In this section, we have set our notations and formulated the functional setting for \eqref{1}, which are used throughout the paper.
For every $s\in (0,1)$, we recall the fractional Sobolev spaces as below
$${H^{s}(\mathbb{R}^n)} :=  \left\{ u \in L^{2}(\mathbb{R}^n):~~\frac{|u(x) - u(y)|}{|x - y|^{\frac{n}{2} + s}} \in L^{2}({\mathbb{R}^n}\times {\mathbb{R}^n)} \right\} $$ which is contained in $H^1(\mathbb R^n)$. We assume that $\Omega \cup \Pi_2$ is bounded with a smooth boundary. 
The symbol $D$ is used throughout the article instead of $(\Omega \cup {\Pi_2} \cup (\partial\Omega\cap\overline{\Pi_2}))$ to keep things simple.
  We define the function space  $\mathcal{X}^{1,2}(D)$ as 
\begin{align*}
    \mathcal{X}^{1,2}(D)  = \{u\in H^1(\mathbb{R}^n) : ~u|_{D} \in H^1_0(D) ~\text{and}~ u \equiv 0~ a.e. ~\text{in}~ {D^c}\}, ~\text{where} ~D^c=\Pi_1\cup (\partial \Omega \cap \overline{\Pi_1}),
\end{align*}
 equipped with the following norm
 $$ \eta(u)^2 =  ||\nabla{u}||^2_{L^{2}(\Omega)}+ [u]^2_{s},~~\text{for}~ u\in \mathcal{X}^{1,2}(D), $$
 where  $[u]^2_s$ is the Gagliardo seminorm of $u$ defined by 
 $$[u]^2_{s} = C_{n,s}~ \left(\int_{D_{\Omega}} \frac{|u(x)-u(y)|^{2}}{|x-y|^{n+2s}} \, dxdy \right), ~{D_\Omega =\mathbb{R}^{2n} \backslash (\Omega^{c} \times \Omega^{c})}.$$ 
}
The following {Poincaré inequality} can be established following the arguments of Proposition 2.4 in \cite{MR4065090} and taking advantage of partial Dirichlet boundary conditions in $D^c$.
{
\begin{Proposition}\label{Poin} (Poincar\'e type inequality) There exists a constant $C=C(\Omega, n,s)>0$ such that
$$
\dint_{\Omega}| u|^2\,dx\leq C\left(\int_{\Omega} |\nabla u|^2\,dx+  \int_{D_{\Omega}} \dfrac{|u(x)-u(y)|^2}{|x-y|^{n+2s}}\,dxdy\right),
$$
for every  $u\in\mathcal{X}^{1,2}(D)$, i.e. $\|u\|^2_{L^2(\Omega)}\leq C \eta(u)^2$. 
\end{Proposition}
 It is easy to check that $\eta(.)$ is a norm on ${\mathcal{X}^{1,2}(D)}$, since $\eta(u)$ = 0 implies  $u=0$  a.e. in ~$\mathbb{R}^n$ which follows straightaway from Proposition \ref{Poin}.}

We now describe a few essential properties of this space. 
\begin{Proposition}\label{p2.2}
The space $\left(\mathcal{X}^{1,2}(D), \langle. , .\rangle\right)$ is a Hilbert space with scalar product given by
$$\langle{ u},{ v}\rangle := \int_{\Omega} \nabla u\cdot \nabla{v} \,dx + \frac{C_{n,s}}{2} \int_{D_{\Omega}} {\dfrac{(u(x)-u(y)) (v(x)-v(y))}{|x-y|^{n+2s}}} ~ dx dy. $$
\end{Proposition}
\begin{proof}
    In order to show that ${\mathcal{X}^{1,2}(D)}$ is a Hilbert space, we need to prove that ${\mathcal{X}^{1,2}(D)}$ is complete with respect to the norm $\eta(.)$. For this, let $\{u_j\}_{j\in {\mathbb{N}}}$ be a Cauchy sequence in ${\mathcal{X}^{1,2}(D)}$. By using Proposition \ref{Poin}, we can easily deduce that $\{{u_j\}_{j\in {\mathbb{N}}}}$ is a Cauchy sequence in $L^{2}(\Omega)$ and since it is complete Banach space, there exist{s} a $u\in L^{2}(\Omega)$ such that $u_{j} \to u$  in $L^{2}(\Omega)$ as $j\to \infty$. Hence, up to a subsequence, still denoted by itself, $u_{j} \to u$ a.e. in $\Omega$, for this, we refer [\cite{MR697382}, Theorem IV.9 ]. Clearly, we also have that $\{\nabla u_j\}_j$ is a Cauchy sequence in $L^2(\Omega)$, and hence there exist{s} $w\in L^2(\Omega)$ such that $\nabla u_j \to w$ in $L^2(\Omega)$ as $j \to \infty$.
Now we will show that $\nabla u = w$.
If we fix $\phi \in C^{\infty}_0(\Omega)$, then by the definition of weak derivative one has
\begin{equation}\label{eq2.6}
 \int_{\Omega} \frac{\partial u_j}{\partial x_{i}}\phi~ \mathrm{d}x= -\int_{\Omega}u_j\frac{\partial \phi}{\partial x_{i}}~\mathrm{d}x,~\forall~1\leq i\leq n .    
\end{equation}
Using the fact that strong convergence in $L^{2}(\Omega)$ implies weak convergence in these spaces, we have
\begin{equation}\label{eq2.7}
 \int_{\Omega}u_j\frac{\partial \phi}{\partial x_{i}}~\mathrm{d}x\to \int_{\Omega}u\frac{\partial \phi}{\partial x_{i}}~\mathrm{d}x~\text{and~}\int_{\Omega} \frac{\partial u_j}{\partial x_{i}}\phi~ \mathrm{d}x\to\int_{\Omega} w_i\phi~ \mathrm{d}x~\text{as}~j\to\infty.    
\end{equation}
Letting $j\to\infty$ in \eqref{eq2.6} and using \eqref{eq2.7}, we get
 $$ \int_{\Omega} w_i\phi~ \mathrm{d}x= -\int_{\Omega}u\frac{\partial \phi}{\partial x_{i}}~\mathrm{d}x,~\forall~1\leq i\leq n.$$
 It follows at once that $$\frac{\partial u}{\partial x_{i}}=w_i\in L^{2}(\Omega),~\forall~1\leq i\leq n,~i.e.,~\nabla u=w.$$  Hence, the proof of our claim is finished. Next, we aim to prove that $\mathcal{X}^{1,2}(D)$ is complete. For this, one can notice that $  u_j \to u ~\text{a.e.}~\text{in}~\Omega$ as $j\to\infty$. More precisely, it means that there exists a set $B_1 \subset \Omega$ such that
\begin{equation}\label{D1}
\left|B_1\right|=0 \quad \text { and } \quad u_j(x) \rightarrow u(x) ~\text{as}~j\to\infty\quad \text { for all } x \in \Omega \backslash B_1.
    \end{equation}
Furthermore, given any $\mathcal{H}: \mathbb{R}^n \rightarrow \mathbb{R}$, for any $(x, y) \in \mathbb{R}^{2 n}$, we consider the following function
\begin{equation}\label{G}
G_{\mathcal{H}}(x, y)=\sqrt{C_{n,s}}~\left[\frac{(\mathcal{H}(x)-\mathcal{H}(y)) \mathbf{1}_{D_{\Omega}}(x, y)}{|x-y|^{\frac{n+2 s}{2}}}\right].   
\end{equation}
Now, since
$$
G_{u_j}(x, y)-G_{u_k}(x, y)=\sqrt{C_{n,s}}~\left[\frac{\left(u_j(x)-u_k(y)-u_j(x)+u_k(y)\right) \mathbf{1}_{D_{\Omega}}(x, y)}{|x-y|^{\frac{n+2 s}{2}}}\right]
$$
and $\{u_j\}_j$ is a Cauchy sequence, we have for any $\varepsilon>0$, there exists $n_{\varepsilon} \in \mathbb{N}$ such that, if $j, k \geq n_{\epsilon}$ then
$$
\epsilon^2 \geq  C_{n,s}~ \int_{D_{\Omega}} \frac{\left|\left(u_j-u_k\right)(x)-\left(u_j-u_k\right)(y)\right|^2}{|x-y|^{n+2 s}} d x d y =\left\|G_{u_j}-G_{u_k}\right\|_{L^2\left(\mathbb{R}^{2 n}\right)}^2.
$$
It follows that that $\{G_{u_j}\}_j$ is a Cauchy sequence in $L^2\left(\mathbb{R}^{2 n}\right)$. From this, we infer that there exists a $G\in L^2\left(\mathbb{R}^{2 n}\right)$ such that $G_{u_j}\to G$ in $L^2\left(\mathbb{R}^{2 n}\right)$ as $j\to\infty$. Hence, without loss of generality, we have $G_{u_j}\to G$ a.e. in $\mathbb{R}^{2 n}$ as $j\to\infty$. It means that we can find $B_2 \subset \mathbb{R}^{2 n}$ such that
\begin{equation}\label{D2}
\left|B_2\right|=0 ~\text{and}~~
 G_{u_j}(x, y) \rightarrow G(x, y)~\text{as} ~j\to\infty,~\forall~(x, y) \in \mathbb{R}^{2 n} \backslash B_2.
    \end{equation}
 For any $x \in \Omega$, we define the following sets
 $$  M_x =\left\{y \in \mathbb{R}^n:(x, y) \in \mathbb{R}^{2 n} \backslash B_2\right\},~~~P  =\left\{x \in \Omega:\left|\mathbb{R}^n \backslash M_x\right|=0\right\} $$
and $$ N =\left\{(x, y) \in \mathbb{R}^{2 n}: x \in \Omega \text { and } y \in \mathbb{R}^n \backslash M_x\right\}.$$
Our next goal is to show 
\begin{equation}\label{N}
N \subseteq B_2
    \end{equation}
Indeed, if $(x, y) \in N$, then $y \in \mathbb{R}^n \backslash M_x$, namely $(x, y) \notin \mathbb{R}^{2 n} \backslash B_2$ and hence $(x, y) \in B_2$, as desired \eqref{N}. In addition, by \eqref{D2} and \eqref{N}, we find that $|N|=0$.
Hence, by Fubini's Theorem, it follows that
$
0=|N|=\int_{\Omega}\left|\mathbb{R}^n \backslash M_x\right| d x
$
and thus $\left|\mathbb{R}^n \backslash M_x\right|=0$ for a.e. $x \in \Omega$. Also, we have $|\Omega \backslash P|=0$ which, together with \eqref{D1}, gives
$$
\left|\Omega \backslash\left(P \backslash B_1\right)\right|=\left|(\Omega \backslash P) \cup B_1\right| \leqslant|\Omega \backslash P|+\left|B_1\right|=0 .
$$
In particular, we infer that $P \backslash B_1$ is non empty.
Let us fix $x_0 \in P \backslash B_1$. Now, since $x_0 \in \Omega \backslash B_1$, we have
$$
\lim _{j \rightarrow+\infty} u_j\left(x_0\right)=u\left(x_0\right)
$$
by \eqref{D1}. Moreover, $\left|\mathbb{R}^n \backslash M_{x_0}\right|=0$,  since $x_0 \in P$, namely for any $y \in M_{x_0}$, it follows that $\left(x_0, y\right) \in \mathbb{R}^{2 n} \backslash B_2$. Hence, by using \eqref{G} and \eqref{D2}, we obtain that
$$
\lim _{j \rightarrow+\infty} G_{u_j}\left(x_0, y\right)=\sqrt{C_{n,s}}~\left|x_0-y\right|^{\frac{-(n+2 s)}{2}} \lim _{j\rightarrow+\infty}\left(u_j\left(x_0\right)-u_j(y)\right) \mathbf{1}_{D_{\Omega}}\left(x_0, y\right)=G\left(x_0, y\right).
$$
In addition, since $\Omega \times\left(\mathbb{R}^N \backslash \Omega\right) \subseteq D_{\Omega}$, by the definition in \eqref{G},
$$
G_{u_j}\left(x_0, y\right)=\sqrt{C_{n,s}}\left[\frac{u_j\left(x_0\right)-u_j(y)}{\left|x_0-y\right|^{\frac{n+2s}{2}}}\right] \quad \text { for a.e. } y \in \mathbb{R}^n \backslash \Omega.
$$
Hence, we have
$$
\begin{aligned}
\lim _{j \rightarrow+\infty} u_j(y) & =\lim _{j \rightarrow+\infty}\left(u_j\left(x_0\right)-\sqrt{\frac{1}{C_{n,s}}}~~\left|x_0-y\right|^{\frac{n+2s}{2}} G_{u_j}\left(x_0, y\right)\right) \\
& =u\left(x_0\right)-\sqrt{\frac{1}{C_{n,s}}}~~\left|x_0-y\right|^{\frac{n+2s}{2}} G\left(x_0, y\right)~\text{a.e.}~ y\in \R^n\setminus \Omega.
\end{aligned}
$$
This implies that $u_j\to u$ a.e. in $ \mathbb{R}^n \backslash \Omega$ as $j\to\infty$. This and  \eqref{D1} say that $u_j$ converges a.e. in $\mathbb{R}^n.$
Up to the change of notation, we will say that $u_j$ converges a.e. in $\R^n$ to some $u.$
Consequently, by the Fatou's {L}emma, we obtain
\begin{align*}
C_{n,s}~\int_{D_{\Omega}} \frac{|u(x)-u(y)|^2}{|x-y|^{n+2s}}
dxdy &\leq C_{n,s}~\liminf_{j\to \infty}\int_{D_{\Omega}} \frac{|u{_{{j}}}(x)-u_{{j}}(y)|^2}{|x-y|^{n+2s}}
dxdy\\
&\leq C_{n,s}~ \liminf_{j\to \infty}\int_{D_{\Omega}} \frac{|u{_{{j}}}(x)-u_{{j}}(y)|^2}{|x-y|^{n+2s}}
\,dxdy + \liminf_{j\to \infty}\int_{{\Omega}} |\nabla u{_{{j}}}|^2 dxdy\\
& =\liminf_{j\to \infty}\eta(u{_{{j}}})^2<+ \infty.
\end{align*}
Hence, we deduce that $[u]^2_s <+\infty$. Now it remains to show that $\eta(u_j) \to \eta(u)$ as $j \to \infty.$
  For this, let us take $ i\geq n_{\epsilon}$, then by using Fatou's
{L}emma, we get
\begin{align*}
[u_i-u]^2_s \leq \liminf_{j \to \infty}[u_i-u{_{{j}}}]^2_s \leq \liminf_{j \to \infty}[u_i-u{_{{j}}}]^2_s + \liminf_{j \to \infty} \|\nabla u_i - \nabla u{_{{j}}}\|^2_{L^2(\Omega)}\leq \liminf_{j\to \infty}\eta(u_i-u{_{{j}}})^2\leq  \epsilon.
\end{align*}
Hence, $u_i \to u \in {\mathcal{X}^{1,2}(D)}$ as $i \to \infty$, which completes the proof.
\end{proof}
\begin{Proposition}\label{P}
Let s~$\in (0,1)$ then for every $ u,v\in {C^\infty_0(D)}$, it holds
\begin{align*}
    \int_{\Omega}v \mathcal{L} u \,dx  
    &= \int_{{\Omega}} \nabla u \cdot\nabla{v} \,dx +  \frac{C_{n,s}}{2}\int_{D_{\Omega}} {\dfrac{(u(x)-u(y)) (v(x)-v(y))}{|x-y|^{n+2s}}} ~ dx dy\\
    &-  \int_{\partial \Omega\cap\overline{\Pi_2}} v {\frac{\partial u}{\partial \nu}}~ d{\sigma}-  \int_{\Pi_2} v {\mathcal{N}}_s u~ dx.
    \end{align*}
    \end{Proposition}
    \begin{proof}
        By directly using the integration by parts formula and the fact that $u,v \equiv 0$ a.e. in ${\Pi_1}\cup (\partial \Omega \cap \overline{\Pi_1})$, we can follow {L}emma 3.3 of \cite{MR3651008} to obtain the conclusion.
    \end{proof}
{
     \begin{Corollary}
Since $ C^\infty_0(D)$ is dense in $\mathcal{X}^{1,2}(D)$, so Proposition \ref{P} still holds for functions in $\mathcal{X}^{1,2}(D)$.   
\end{Corollary}
}
\begin{Definition}\label{new5}
We say that $u\in \mathcal{X}^{1,2}(D)$ is a weak solution to the problem \eqref{1} if 
\begin{equation}\label{dd2.2}
   \int_{{\Omega}} \nabla u\cdot\nabla \varphi \,dx +\frac{C_{n,s}}{2}\int_{D_{\Omega}} \frac{(u(x)-u(y))(\varphi(x)-\varphi(y))}{|x-y|^{n+2s}} dxdy  = {\lambda}\int_{\Omega} u^q\varphi\,dx+\int_{\Omega} u^p\varphi \,dx~~\forall  ~\varphi \in {\mathcal{X}^{1,2}(D)}\\
\end{equation}
\end{Definition}
  
  {For functions in $\mathcal{X}^{1,2}(D)$, we now present an embedding result which is a consequence of $\mathcal{X}^{1,2}(D)\hookrightarrow H^1(\mathbb R^n)$ and Sobolev embedding of $H^1(\mathbb R^n)$. 
\begin{remark}\label{l1}
   For $D$ is bounded (since $\Omega\cup\Pi_2$ is bounded) with smooth boundary, we have 
    $$\mathcal{X}^{1,2}(D)\hookrightarrow \hookrightarrow L^r_{loc}(\R^n)$$
    compact embedding, for $r\in [1,2^*)$ and continuous embedding for $r\in [1, 2^*],$ where $2^*=\frac{2n}{n-2}$.
\end{remark}} 
Now, for $a\in \mathbb R$, we consider the standard truncation functions given by
\begin{equation}\label{new0}
    M_a(u)= \max \{-a, \min\{a,u\}\},~~ K_a(u)= u- M_a(u).
\end{equation} 
 The following properties of space ${\mathcal{X}^{1,2}(D)}$ shall be helpful to get regularity results for elliptic problems in space ${\mathcal{X}^{1,2}(D)}$.

\begin{Proposition} \label{p2.5}
{Assume $u\in {\mathcal{X}^{1,2}(D)}$, then the following holds.} 
\begin{enumerate}
    \item  Let $\phi$ be  Lipschitz in $\mathbb{R}$ such that $\phi(0)=0$, then $\phi(u)
\in  {\mathcal{X}^{1,2}(D)}$. 
In particular for any $a>0$, $M_a(u),\\ K_a(u) \in {\mathcal{X}^{1,2}(D)}$.
\item For any $a\geq 0$,
{$$ \eta(K_a(u))^2\leq \int_{\Omega} K_a(u) \mathcal{L} u\, dx + \int_{\overline{\Pi_2}\cap \partial \Omega} K_a(u) \frac{\partial u}{\partial \nu}\,d{\sigma}+\int_{\Pi_2}K_a(u) \mathcal{N}_s(u) \,dx,$$}
{\item For any $a\geq 0$,
{$$ \eta(M_a(u))^2\leq \int_{\Omega} M_a(u) \mathcal{L} u\, dx + \int_{\overline{\Pi_2}\cap \partial \Omega} M_a(u) \frac{\partial u}{\partial \nu}\,d{\sigma}+\int_{\Pi_2}M_a(u) \mathcal{N}_s(u) \,dx.$$
}}
\end{enumerate}
\end{Proposition}
\begin{proof} We can easily prove $(1)$ with the help of arguments in {P}roposition 3 in \cite{MR3393266}. For proving $(2)$ and $(3)$, we claim that, for any arbitrary $b, d \geq 0$ and for any $x \in \mathbb{R}^n,$
\begin{equation} \label{P1}
 b({K_a(u)}\mathcal{L}(M_a(u)))(x)+d \left( {K_a(u)} \frac{\partial (M_a(u))}{\partial \nu}+ {K_a(u)}\mathcal{N}_s(M_a(u))\right)(x) \geq 0.
 \end{equation}
We can check that if $x$ is such that $ K_a(u)(x) =0$ then equation \eqref{P1} is obvious. So, let  $x$ is such that $ K_a(u)(x) \neq 0$ then either $K_a(u)(x)> 0$ or $K_a(u)(x)< 0$. { If $ K_a(u)(x)> 0$ then $M_a(u)(x)= a$} which is maximum value that $M_a(u)$ attains and therefore  $\mathcal{L}(M_a(u))(x)\geq 0$, $\frac{\partial (M_a(u))}{\partial \nu}(x)\geq 0$ and $\mathcal{N}_s(M_a(u))(x) \geq 0.$ When $ K_a(u)(x)<0$, we have that $M_a(u)(x)= -a$ which is minimum value that $M_a(u)$ attains and therefore  $\mathcal{L}(M_a(u))(x)\leq 0$, $\frac{\partial (M_a(u))}{\partial \nu}(x)\leq 0$ and $\mathcal{N}_s(M_a(u))(x) \leq 0.$ So we can easily conclude the claim \eqref{P1}. By using equation \eqref{P1} and {P}roposition \ref{P}, it follows that
\begin{equation} \label{2.3}
 \begin{split}
\int_{\Omega} M_a(u) \mathcal{L}(K_a(u))dx + \int_{\overline{\Pi_2}\cap\partial\Omega} M_a(u)\frac{\partial (K_a(u))}{\partial \nu}d{\sigma}+\int_{\Pi_2} M_a(u)\mathcal{N}_s (K_a(u))dx\\
= \int_{\Omega} K_a(u) \mathcal{L}(M_a(u))dx + \int_{\overline{\Pi_2}\cap\partial\Omega} K_a(u) \frac{\partial (M_a(u))}{\partial \nu}d{\sigma}+\int_{\Pi_2} K_a(u) \mathcal{N}_s(M_a(u))dx\geq 0.
\end{split}   
\end{equation} 
Now, by using \eqref{P1}  condition and  {P}roposition \ref{p2.2} and Proposition \ref{P}, we obtain 
\begin{equation} \label{2.4}
\begin{split}
\eta(K_a(u))^2 &=    \int_{\Omega} |\nabla K_a(u)|^2 dx +{\frac{C_{n,s}}{2}}\int_{D_{\Omega}} \frac{(K_a(u)(x)-K_a(u)(y))^2}{|x-y|^{n+2s}} dxdy\\
&= \int_{\Omega}K_a(u) \mathcal{L}(K_a(u))dx + \int_{\overline{\Pi_2}\cap\partial\Omega}K_a(u)\frac{\partial (K_a(u))}{\partial \nu} \,d{\sigma}+ \int_{\Pi_2} K_a(u)\mathcal{N}_s K_a(u)dx\\
&=\int_{\Omega} K_a(u) \mathcal{L}(u-M_a(u))dx + \int_{\overline{\Pi_2}\cap\partial\Omega} K_a(u)\frac{\partial (u-M_a(u))}{\partial \nu} \,d{\sigma}+\int_{\Pi_2} K_a(u)\mathcal{N}_s(u-M_a(u))dx 
\end{split}
    \end{equation}
    Similarly, we can obtain
    \begin{equation} \label{2.5}
    \begin{split}
 \eta(M_a(u))^2 
 &= \int_{\Omega} M_a(u) \mathcal{L}(u- K_a(u))dx + \int_{\overline{\Pi_2}\cap\partial\Omega}M_a(u)\frac{\partial (u- K_a(u))}{\partial \nu} \,d{\sigma}+ \int_{\Pi_2} M_a(u)\mathcal{N}_s(u- K_a(u))dx.
 \end{split}
\end{equation}
Hence, the claim in $(2)$ follows from equations \eqref{2.4} and  \eqref{2.3} as well as the claim $(3)$ follows from \eqref{2.5} and \eqref{2.3}.
\end{proof}

\section{Auxiliary problem}
In this section, we study an elementary  problem given by
\begin{equation} \label{2.6}
    \left\{\begin{split} \mathcal{L}u\: &= g~~ \text{in} ~\Omega, \\
      u&=0~~\text{in} ~~{D^c},\\
 \mathcal{N}_s(u)&=0 ~~\text{in} ~~{\Pi_2}, \\
 \frac{\partial u}{\partial \nu}&=0 ~~\text{in}~~ \partial \Omega \cap \overline{\Pi_2}.
    \end{split} \right.
\end{equation}
where $\Omega$ is a open domain of $\mathbb{R}^n$, $g\in (\mathcal{X}^{1,2}(D))^*$ and $(\mathcal{X}^{1,2}(D))^*$ is the dual space of ${\mathcal{X}^{1,2}(D)}$.
\begin{definition}
We say that $u\in {\mathcal{X}^{1,2}(D)}$ is a weak solution to the problem \eqref{2.6} if 
\begin{equation} \label{2.7}
    \int_{\Omega} \nabla u\cdot\nabla \varphi ~dx + \frac{C_{n,s}}{2}\int_{D_{\Omega}} \frac{(u(x)-u(y))(\varphi(x)-\varphi(y))}{|x-y|^{n+2s}} dxdy  = \int_\Omega g\varphi~dx, ~ \forall~\varphi \in {\mathcal{X}^{1,2}(D)}.
\end{equation}
\end{definition}
\subsection{Case: \texorpdfstring{$g=g(x)$}{}} 
It is to be noted that the Lax-Milgram Theorem implies the existence and uniqueness of energy solutions to problem \eqref{2.6}, see Theorem 1.1 in \cite{MR4387204}. Additionally, if $g \geq 0$ a.e in $\Omega$, then $u \geq 0$ a.e in $\R^n$. Indeed, for a weak solution $u \in {\mathcal{X}^{1,2}(D)}$ of \eqref{2.6}, we know that $u^{-} = \max(- u, 0) \in {\mathcal{X}^{1,2}(D)}$ and using $u^{-}$ as a test function in equation \eqref{2.7} then from {P}roposition \ref{p2.5}, we easily get $u^{-} = 0$ a.e. in $\mathbb R^{n}$.
\begin{definition}
    For each non-negative test function in ${\mathcal{X}^{1,2}(D)}$, a super solution (or subsolution) of \eqref{2.6} is a function in ${\mathcal{X}^{1,2}(D)}$ which satisfies \eqref{2.7} with equality substituted by $"\geq"$ (respectively $"\leq"$), for each non-negative test function in ${\mathcal{X}^{1,2}(D)}$.
\end{definition} 
\begin{lemma}\label{l2.7}
Suppose that problem \eqref{2.6} has a subsolution $\underline{z}$ and a supersolution $\bar{z}$, satisfying $\underline{z}\leq \bar{z}$ a.e. in $\R^n$. Then there exists a weak solution $z$  of \eqref{2.6} satisfying $\underline{z}\leq z \leq \bar{z}$ a.e. in $\R^n$.
\end{lemma}
\begin{proof}
    This result can be established using a standard iterative argument and precisely following   {L}emma 2.2 in \cite{MR1964476}.
\end{proof}
The following regularity result is a straightforward application of Proposition \ref{p2.5}.
\begin{lemma} \label{l2.8}
Let $u$ be a solution of the problem \eqref{2.6} and $g$ be any function defined as in \eqref{2.6} satisfying  $ g \in L^m({{D}})$ with $m>\frac{n}{2s}\left(>\frac{n}{2}\right)$, then $u \in L^{\infty}({D}).$
\end{lemma}
\begin{proof}
We borrow some arguments, from \cite{MR3393266}  to establish the proof.  Let $a>0$ be any arbitrary real number. Now, taking $\varphi=K_a(u)$ as a test function in \eqref{2.6}, where $K_a(u)
$ is defined in \eqref{new0}, we obtain 
 $$\int_{\Omega} \nabla u\cdot\nabla K_a(u(x)) dx +\int_{D_{\Omega}} \frac{(u(x)-u(y))(K_a(u(x))-K_a(u(y)))}{|x-y|^{n+2s}} dxdy = \int_{\Omega}g K_a(u(x))\,dx.$$
 On the other hand, we have {$u(x)=K_a(u(x))+M_a(u(x))$}, then using Proposition \ref{p2.5}(2), we can write
\begin{align*}
\eta(K_a(u(x)))^2 & \leq \int_{\Omega}g K_a(u(x)) ~dx+ \int_{\overline{\Pi_2}\cap\partial\Omega} K_a(u(x)) \frac{\partial u}{\partial \nu} \,d{\sigma} + \int_{\Pi_2}  K_a(u(x))\mathcal{N}_s u \,dx\\
& \leq \int_{D}{\left|g  K_a(u(x)) \right|}~dx+ \int_{\overline{\Pi_2}\cap\partial\Omega} K_a(u(x)) \frac{\partial u}{\partial \nu} \,d{\sigma} + \int_{\Pi_2}  K_a(u(x))\mathcal{N}_s u \,dx\\
&  =\int_{A(a)} {\left|g K_a(u(x))\right|}~ dx +  \int_{\overline{\Pi_2}\cap\partial\Omega}  K_a(u(x)) \frac{\partial u}{\partial \nu} \,d{\sigma}+\int_{\Pi_2}  K_a(u(x)) \mathcal{N}_s u ~ dx,
\end{align*} 
where $A(a) = \{ x\in {D} : |u(x)|\geq a\}$ and also we used the fact that $K_a(u)=0$ on $A(a)^c$. 
Since, $u$ satisfies the mixed boundary conditions given in \eqref{2.6}, then we obtain
$$\eta( K_a(u))^2\leq \int_{A(a)} |g K_a(u(x)) |~dx$$
Now, using Remark \ref{l1} and H\"older inequality, we obtain
$$ \|K_a(u)\|^2_{L^{{2}^*}({D})}\leq S \eta(K_a(u))^2\leq S\|g\|_{L^m({D})}\|K_a(u)\|_{L^{{2}^*}({D})} |A(a)|^{\left(1-\frac{1}{2^*}-\frac{1}{m}\right)},$$
for some $m> 1$ and $S>0$. After simplifying this, we obtain 
$$ \|K_a(u)\|_{L^{{2}^*}({D})}\leq S \|g\|_{L^m({D})} |A(a)|^{\left(1-\frac{1}{2^*}-\frac{1}{m}\right)}.$$
Now, it is easy to check that when $a < h$ then $A(h)\subset A(a)$ and $K_a(u)\mathbf{1}_{A(h)}\geq (h-a)$, where $\mathbf{1}_{A(h)}$ is characteristic function on $A(h)$. Then we have
$$(h-a)|A(h)|^{{\frac{1}{2^*}}}\leq S\|g\|_{L^m({D})} |A(a)|^{\left(1-\frac{1}{2^*}-\frac{1}{m}\right)}$$ which implies
$$|A(h)|\leq S^{2^*}\frac{\|g\|^{2^{*}}_{L^m({D})}|A(a)|^{2^*\left(1-\frac{1}{2^*}-\frac{1}{m}\right)}}{(h-a)^{2^*}}$$
Let us choose $m>\frac{n}{2s}>\frac{n}{2}$, then we have that
$$2^*\left(1-\frac{1}{2^*}-\frac{1}{m}\right)>1.$$
Hence, if we apply {L}emma 14 in \cite{MR3393266} with the choice of $\psi(a_0)=|A(a_0)|$ then we obtain $\psi(a_0)\equiv 0$,  where {$a_0^{2^*}= S^{2^*}\frac{\|g\|^{2^{*}}_{L^m({D})}|A(a)|^{2^*\left(1-\frac{1}{2^*}-\frac{1}{m}-1\right)}}{(h-a)^{2^*}} 2^{\frac{\delta \alpha}{\alpha-1}}$ and $\delta={2^*\left(1-\frac{1}{2^*}-\frac{1}{m}-1\right)} $, $\alpha=2^*$}. This implies $|A(a_0)|=0$ that is ${|u|}\leq a_0$ a.e. in ${D}$.  
Thus, $u\in {L^{\infty}({D})}$.
\end{proof}
\begin{theorem}\label{W2p}
Suppose $g \in L^{\infty}(D)$ and  $\partial D$ is of class $C^{1,1}$.  Also let  $u \in \mathcal{X}^{1,2}(D)$ be weak solution of \eqref{2.6} then
$ u \in C^{1, \gamma}(\bar\Omega)$, for some $\gamma \in(0,1)$. Moreover, $u\in C^{0,\beta}(\R^n)$, for some $\beta\in (0,1).$
    \end{theorem}
\begin{proof}
It is easy to see that if $g\in L^{\infty}(D)$ then $g \in L^p(D)$, $\forall$~ $p \geq 2$, since $D$ is bounded. We utilize  [Theorem 3.1.22 in \cite{garroni2002second}] to achieve $W^{2, p}$ regularity for some $p > 1$ i.e. $\|u\|_{W^{2,p}(\Omega)}\leq C_1\|g\|_{L^p(\Omega)}$ and combining this with  compact embedding of $W^{2, p}(\Omega)$ in $C^{1,\gamma}(\bar \Omega)$ for $p > n$, we get $ u \in C^{1, \gamma}(\bar\Omega)$. Furthermore, from $\mathcal{N}_s u=0$ in $\Pi_2$,  we have
$$ u(x)\int_{\Omega} \frac{dy}{|x-y|^{n+2s}}= \int_{\Omega} \frac{u(y) dy}{|x-y|^{n+2s}}.
$$
 In the above equality, we observe that both parametrized integrals with respect to $x$ are differentiable in $\Pi_2$, since for any $x\in \Pi_2$, $\text{dist} (x,\Omega)>0$. This implies $u$ is differentiable in  $\Pi_2.$ Next, from this, $u \in C^{1, \gamma}(\bar\Omega)$  and $u=0$ in $\Pi_1$, we may obtain that $u \in C^{0,\beta}(\mathbb{R}^n).$
\end{proof}

\subsection{ Case: \texorpdfstring{$g=g(x,t)$}{}}
 
Let us now state the  following conditions which shall be called periodically in subsequent results of this subsection:
\begin{enumerate}
\item $g: \Omega \times [0,\infty)\to \R$ is a Carath\'eodory function.
\item $g(\cdot,t)\in L^{\infty}(\Omega)$, for every $t\geq0.$
\item $|g(x,t)|\leq c(1+|t|^{p}), ~\forall~t\geq 0,  \text{where} ~1< p\leq \frac{n+2}{n-2}=2^*-1$.
\item The function $t\mapsto \frac{g(x,t)}{t}$ is strictly decreasing in $(0,\infty).$
\end{enumerate}

We can establish the following {T}heorem as an extension of {L}emma \ref{l2.8} using [Theorem 4.1 in \cite{biagi2021brezis}].
\begin{theorem}\label{t2.10}
Assume $u$ to be {a} weak solution to \eqref{2.6} with g satisfying the growth condition $(3)$, then $u \in L^{\infty}({{D}})$.
\end{theorem}
The following result is a version of the strong maximum principle that can be proved using arguments in [Theorem 3.1 in \cite{biagi2021brezis}]. 
From $2.$, $4.$ stated above, we have 
$$(5):~~~~~g(x,t)\geq -c_g t, \text{for a.e.}~ x\in\Omega, ~\text{for every}~ 0<t<1.$$
From this and $(1)$, we have 
$$(6):~~~~~g(x,0)\geq 0 ~\text{for a.e.}~ x\in\Omega.$$
\begin{theorem}\label{strb}
Suppose $g$ satisfies $(1), (2), (4)$ and 
 $|g(x,t)|\leq c(1+|t|), ~\forall~t\geq 0$.
Also assume
$u\in \mathcal{X}^{1,2}(D)$ is  non negative in $\mathbb R^n$ which satisfies Definition \eqref{2.7} 
then either $u\equiv 0$ or $u>0$ a.e. in $\Omega.$
\end{theorem}
\begin{proof}
Let $x_0\in \Omega$ and $R>0$ is such that $B_{x_0}(R) \subset \Omega ~$(\text{where} $B_{x_0}(R)=\{x\in \Omega: |x- x_0|< R\})$.\\
\textbf{Claim(1):} If $u \equiv 0 ~ \text{a.e. on}~ B_{x_0}(R)$, then $u \equiv 0 ~\text{a.e. on}~ \Omega.$\\
Indeed, taking a test function $0\leq \varphi \in C_0^\infty(B_{x_0}(R))$ in \eqref{dd2.2}  satisfying the following properties
\begin{equation}\label{maxp1}
\int_{B_{x_0}(R)} \varphi \, dx = 1,
\end{equation}
then we have that
\begin{equation}\label{eq2.10}
\begin{aligned}
 \int_{B_{x_0}(R)} g(x, 0) \varphi \, dx &= \int_{\Omega} g(x, u) \varphi \, dx = \int_{\Omega}\nabla u\cdot\nabla \varphi\, dx+\int_{D_\Omega } \frac{(u(x) - u(y))(\varphi(x) - \varphi(y))}{|x - y|^{n+2s}} \, dx \, dy  \\
&=\left(\int_{B_{x_0}(R)\times B_{x_0}(R)}+ \int_{B_{x_0}(R)\times \Omega \setminus B_{x_0}(R)}+\int_{\Omega \setminus B_{x_0}(R) \times B_{x_0}(R)}\right) \frac{(u(x) - u(y))(\varphi(x) - \varphi(y))}{|x - y|^{n+2s}} \, dx \, dy  \\
&= -2 \int_{\left(\Omega \setminus B_{x_0}(R)\right)\times {B_{x_0}(R)}}  \frac{u(x) \varphi(y)}{|x - y|^{n+2s}} \, dx \, dy  \leq - \frac{2}{\operatorname{diam}(\Omega)^{n+2s}} \int_{\Omega \setminus B_{x_0}(R)} u(x) \, dx \int_{B_{x_0}(R)} \varphi(y)\, dy\\
&=  - \frac{2}{\operatorname{diam}(\Omega)^{n+2s}} \int_{\Omega \setminus B_{x_0}(R)} u(x) \, dx, ~\text{using \eqref{maxp1}}.
\end{aligned}
\end{equation}
Additionally, since $\varphi$ is nonnegative and $g(x, 0) \geq 0$ a.e. in $\Omega$, it follows that
\begin{equation}\label{eq2.11}
\int_{B_{x_0}(R)} g(x,0) \varphi \,dx\geq 0.
\end{equation}
Combining \eqref{eq2.10} and \eqref{eq2.11}, we obtain
$$0\leq \int_{B_{x_0}(R)} g(x, 0) \varphi \, dx \leq  - \frac{2}{\operatorname{diam}(\Omega)^{n+2s}} \int_{\Omega \setminus B_{x_0}(R)} u(x) \, dx\leq 0$$
which implies
\begin{equation}\label{eq2112}
- \frac{2}{\operatorname{diam}(\Omega)^{n+2s}}\int_{\Omega \setminus B_{x_0}(R)} u(x) \, dx = 0.
\end{equation}
which implies that $u \equiv 0$ a.e. on $\Omega \setminus B_{x_0}(R)$. Recalling that   $u= 0$ a.e. in $ B_{x_0}(R)$, we conclude that $u \equiv 0$ a.e. in $\Omega$.\\
\textbf{Claim(2):} If $u \equiv 0$ in  positive measure subset of $\Omega$, then there exist $x_0 \in \Omega$ and some $R > 0$ such that $B_{x_0}(R)\subset\Omega $ for which $u\equiv 0$ in $B_{x_0}(R)$. \\
Let us assume that $u(x)=0$ for a.e. $x\in \mathcal{B}$, where $\mathcal{B}\subset \Omega$ is a measurable set with $|\mathcal{B}|>0$. Consequently, there exists a $R>0$ such that the ball $B_{x_0}( 2R)\subset \Omega$ and  $|\mathcal{B}\cap B_{x_0}(R)|>0$. 
Next, we choose a nonnegative function $\varphi \in C_0^\infty(B_{x_0}( 2R))$ such that $\varphi \equiv 1$ a.e. in $B_{x_0}(R)$. For every fixed $\epsilon > 0$, we define
$
\varphi_\epsilon = \frac{\varphi^2}{u + \epsilon}
$ 
and  taking  $\varphi_{\epsilon}$  as a test function in \eqref{dd2.2},  we can write
\begin{equation*}
   \int_{\Omega}\nabla u\cdot\nabla \varphi_{\epsilon}\, dx+\int_{D_\Omega } \frac{(u(x) - u(y))(\varphi_{\epsilon}(x) - \varphi_{\epsilon}(y))}{|x - y|^{n+2s}} \, dx \, dy =\int_{\Omega} g(x, u) \varphi_{\epsilon} \, dx 
\end{equation*}
which gives
\begin{equation}\label{eq2.13}
\begin{aligned}
\int_{\Omega} \frac{|\nabla u|^2}{(u + \epsilon)^2} \varphi^2 \, dx \leq \mathcal{H}(u, \varphi, \epsilon)  
 + 2 \int_{\Omega} \frac{|\nabla u| 
|\nabla \varphi|}{(u + \epsilon)} \varphi \, dx - \int_{\Omega} g(x, u) \frac{\varphi^2}{(u + \epsilon)} \, dx,
\end{aligned}
\end{equation}
where, $\mathcal{H}(u, \varphi, \epsilon)=\dint_{D_\Omega } \frac{(u(x) - u(y))(\varphi_{\epsilon}(x) - \varphi_{\epsilon}(y))}{|x - y|^{n+2s}} \, dx \, dy.$
For the first term in the right-hand side of \eqref{eq2.13}, we have the following estimate  as shown in the proof of [\cite{Castro}, Lemma 1.3],
$$
\frac{(u(x) - u(y))(\varphi_\epsilon(x) - \varphi_\epsilon(y))}{|x - y|^{n+2s}} \leq -b \frac{1}{|x - y|^{n+2s}} \left| \log \left( \frac{u(x) + \epsilon}{u(y) + \epsilon} \right) \right|^2 \varphi^2(y) + b \frac{|\varphi(x) - \varphi(y)|^2}{|x - y|^{n+2s}},
$$
for some positive constant $b> 0$. This suggests
$$
\mathcal{H}(u, \varphi, \epsilon) \leq b \int_{D_{\Omega}} \frac{|\varphi(x)-\varphi(y)|^2}{|x-y|^{n+2 s}} d x d y.
$$
Moreover, by applying the weighted Young inequality$\left( \text{for any}~ a,b \geq 0 ~\text{and} ~\delta>0, ~ab\leq \frac{a^2}{2\delta} +\frac{\delta b^2}{2}\right)$ to the second term on the right-hand side of \eqref{eq2.13}, for any $\varepsilon > 0$, we obtain
\begin{equation}\label{eq5.19}
2 \int_{\Omega} \frac{|\nabla u||\nabla \varphi|}{(u+\varepsilon)} \varphi \, dx 
\leq \frac{1}{2} \int_{\Omega} \frac{|\nabla u|^2}{(u+\varepsilon)^2} \varphi^2 \, dx + 2 \int_{\Omega}|\nabla \varphi|^2 \, dx.
 \end{equation}
 Now, using \eqref{eq5.19} in \eqref{eq2.13}, we find
\begin{equation}\label{eq5.20}
\begin{aligned}
\frac{1}{2}\int_{\Omega} \frac{|\nabla u|^2}{(u + \epsilon)^2} \varphi^2 \, dx \leq \mathcal{H}(u, \varphi, \epsilon)  
 + 2 \int_{\Omega}|\nabla \varphi|^2 \, dx - \int_{\Omega} g(x, u) \frac{\varphi^2}{(u + \epsilon)} \, dx.
\end{aligned}
\end{equation}
  For the third integral in the right-hand side of \eqref{eq2.13}, we follow the estimate as in [ \cite{Mugnai}, Lemma 2.4]. By exploiting properties $(5)$, $(6)$ on $g$ (stated above), we get
\begin{equation}\label{eq2.14}
\begin{aligned}
-\int_{\Omega} g(x, u) \frac{\varphi^2}{(u+\varepsilon)} \, dx & = -\int_{\Omega \cap\{u=0\}} g(x, 0) \frac{\varphi^2}{\varepsilon} \, dx 
 \quad -\int_{\Omega \cap\{0<u<1\}} \frac{g(x, u) \varphi^2}{(u+\varepsilon)} \, dx - \int_{\Omega \cap\{u \geq 1\}} \frac{g(x, u) \varphi^2}{(u+\varepsilon)} \, dx \\
& \leq c_g \int_{\Omega \cap\{0<u<1\}} \frac{u}{(u+\varepsilon)} \varphi^2 \, dx + c \int_{\Omega \cap\{u \geq 1\}} \frac{1+u}{(u+\varepsilon)} \varphi^2 \, dx  \leq \left(c_g + 2c\right) \|\varphi\|_{L^2(\Omega)}^2.
\end{aligned}
\end{equation}
Thus, combining \eqref{eq5.20} and \eqref{eq2.14}, we have
\begin{equation}\label{eq520}
\begin{aligned}
\frac{1}{2}\int_{\Omega} \frac{|\nabla u|^2}{(u + \epsilon)^2} \varphi^2 \, dx \leq \mathcal{H}(u, \varphi, \epsilon)  
 + 2 \int_{\Omega}|\nabla \varphi|^2 \, dx +\left(c_g + 2c\right) \|\varphi\|_{L^2(\Omega)}^2,
\end{aligned}
\end{equation}
 employing the chain rule of differentiation and using \eqref{eq520}, we obtain
\begin{equation}\label{eq2.15}
\begin{aligned}
\int_{B_{x_0}(R)} \left| \nabla \log \left(1+\frac{u}{\varepsilon}\right)\right|^2 \, dx = \int_{B_{x_0}(R)} \frac{|\nabla u|^2}{(u+\varepsilon)^2} \, dx \leq \int_{B_{x_0}(R)} \frac{|\nabla u|^2}{(u+\varepsilon)^2} \varphi^2 \, dx \leq K,
\end{aligned}
\end{equation}
where $K = K_{\varphi} > 0$ is a suitable constant independent of $\varepsilon$.
Recalling that $ \mathcal{V}=|\mathcal{B}\cap B_{x_0}(R)|>0$, for any $t > 0$, applying Chebyshev inequality, Poincar\'{e} inequality and \eqref{eq2.15}, we get
\begin{equation}\label{eq2.16}
\begin{aligned}
 \left|\log \left(1+\frac{t}{\varepsilon}\right)\right|^2 \cdot \left|\{u \geq t\} \cap B_{x_0}( R)\right|& \leq \int_{B_{x_0}( R)}\left|\log \left(1+\frac{u}{\varepsilon}\right)\right|^2 \, dx 
\leq K_1 ,
\end{aligned}
\end{equation}
where $K_1=K_1(\Omega, \mathcal{V})>0$.
As a result, because inequality \eqref{eq2.16} is valid for any $\varepsilon > 0$, we can immediately deduce that
$$
\left|\{u \geq t\} \cap B_{x_0}( R)\right|= 0, \quad \forall \, t > 0,
$$
that implies $u \equiv 0$ a.e. in $B_{x_0}( R)$.
Otherwise there exists $t'$ such that $\left|\{u \geq t'\} \cap B_{x_0}( R)\right|>0$. Hence, we get a contradiction if we take the limit as $\epsilon \to 0$ in \eqref{eq2.16}. 
 Thus, the proof is now complete.
\end{proof}
Let us recall the following result from \cite{MR4387204}.
\begin{lemma} \label{new-prop}
Let ${\phi\in C(\R^n, \R)\cap C^2(\Omega,\R)},$ satisfy
$$\int_{\mathbb{R}^n} \frac{|\phi(x)|}{1+|x|^{n+2s}}\,dx<\infty.$$
and solves the following problem
\begin{equation}
    \left\{\begin{matrix} \mathcal{L}\phi\: \geq \:0 ~
     in ~ \Omega,\\
    ~~~~~~~~\phi\geq 0  ~~in ~ \mathbb{R}^{n}\setminus \Omega.
      \\
    \end{matrix} \right.
\end{equation}
Then $\phi\geq 0$ in $\Omega$. Furthermore, if there {exists} some $x_0\in \Omega$ such that $\phi(x_0)=0$, then $\phi \equiv 0$ in $\mathbb{R}^n$.
\end{lemma}
The next corollary is very important and shall be often used in the forthcoming calculations. Its proof is an easy consequence of Lemma \ref{new-prop} or also proving the following result, we can follow the same ideas of Proposition 6.12 in \cite{Cozzi}.
\begin{Corollary} \label{p2.11}
Suppose $g_1, g_2: \mathbb{R}^n \times \mathbb{R}\to \mathbb{R}$ are continuous functions and 
$v_1, v_2 \in {L^{\infty}(\mathbb{R}^n)\cap C^{2s+\xi}}(\mathbb{R}^n)$, for some $\xi>0$ be such that
\begin{equation}\label{max-princ-1}
    \left\{\begin{matrix} \mathcal{L} v_1\: \geq \:g_1(x,v_1) ~
     in ~ \Omega ~~~\\
     \mathcal{L} v_2\: \leq \: g_2(x,v_2) ~
     in ~ \Omega ~~~\\
    v_1\geq v_2  ~~in ~ \mathbb{R}^{n}\setminus \Omega.  \\
    \end{matrix} \right.
\end{equation}
Suppose, additionally that
\begin{equation}\label{e2.18}
    g_2(x,v_2(x))\:\leq\: g_1(x,v_1(x)), ~\text{for any} ~ x\in \Omega.
\end{equation}
Then, if there {exists} a point $x_0\in \Omega$ such that $v_1(x_0)=v_2(x_0)$, then $v_1=v_2$ in $\mathbb R^n$.
\end{Corollary}
We now establish two essential results that will serve our goals, firstly a Picone-type inequality and secondly, Brezis Kamin-type comparison principle.
\begin{theorem}\label{T2.12}
Let $u,v \in {\mathcal{X}^{1,2}(D)}$ and suppose that $\mathcal{L}u\geq 0$ is a bounded {R}adon measure in $\Omega$, { $u>0$ in  $D$ and $\frac{\partial u}{\partial \nu}\geq 0$ on $\overline{\Pi_2}\cap \partial\Omega$}, then 
\begin{equation}\label{eeA}
\int_{\overline{\Pi_2}\cap\partial\Omega}\frac{|v|^2}{u}\frac{\partial u}{\partial \nu} \,d{\sigma}+\int_{\Pi_2}\frac{|v|^2}{u} \mathcal{N}_s u ~dx + \int_{\Omega}\frac{|v|^2}{u} \mathcal{L} u ~dx \leq \eta(v)^2 
   \end{equation}
 \end{theorem}
 \begin{proof}
    Let us recall the definition of $M_a(t)$ from \eqref{new0} and set $M_a(v)=v_a$. We also set $\bar{u}= u+\delta$, for any $\delta>0.$ 
Then by a simple calculation, we get
$\frac{v^2_a}{\bar{u}}\in {\mathcal{X}^{1,2}(D)}$ using Proposition \ref{p2.5}.
From {P}roposition \ref{P}, we have
\begin{equation}\label{e2.19}
\begin{split}
    &\int_{\Pi_2}\frac{|v_a|^2}{\bar{u}} \mathcal{N}_s \bar{u} ~dx + \int_{\overline{\Pi_2}\cap\partial\Omega} \frac{|v_a|^2}{\bar{u}}\frac{\partial \bar{u}}{\partial \nu}\,d{\sigma} +\int_{\Omega}\frac{|v_a|^2}{\bar{u}} \mathcal{L}\bar{u} ~dx \\
    &= \frac{C_{n,s}}{2}\int_{D_{\Omega}} \frac{(\bar{u}(x)-\bar{u}(y))\big(\frac{v_a(x)^2}{\bar{u}(x)}-{\frac{v_a(y)^2}{\bar{u}(y)}\big)}}{|x-y|^{n+2s}} ~ dx dy+ \int_{\Omega} \nabla\bar{u}.\nabla\left(\frac{v_a^2}{\bar{u}}\right)\,dx
    \end{split}
   \end{equation}
   From the classical Picone inequality \cite{MR1618334}, it follows that
   \begin{equation}\label{ee2.20}
   \nabla\bar{u}.\nabla\left(\frac{v_a^2}{\bar{u}}\right)\leq |\nabla v_a|^2
   \end{equation}
   Integrating \eqref{ee2.20} both sides over $\Omega$, we get that
 \begin{equation}\label{ee2.21}
     \int_{\Omega} \nabla\bar{u}.\nabla\left(\frac{v_a^2}{\bar{u}}\right)\,dx \leq \int_{\Omega}|\nabla v_a|^2\,dx=\|\nabla v_a\|^2_{L^2({\Omega})}.
 \end{equation}
On the other hand, by using the nonlocal Picone inequality, see(Proposition 4.2,\cite{MR3273896}), we have
\begin{equation}\label{new2}
    (\bar{u}(x)-\bar{u}(y))\left(\frac{v_a(x)^2}{\bar{u}(x)}-\frac{v_a(y)^2}{\bar{u}(y)}\right)\leq|v_a(x)-v_a(y)|^2, ~\text{for}~ x,y\in \mathbb R^n.
  \end{equation}
   Now, dividing  $|x-y|^{n+2s}$ on both sides of \eqref{new2} and integrating over $D_{\Omega}$, we get
  \begin{equation}\label{ee2.22}
    \frac{C_{n,s}}{2}\int_{D_{\Omega}} \frac{(\bar{u}(x)-\bar{u}(y))\big(\frac{v_a(x)^2}{\bar{u}(x)}-{\frac{v_a(y)^2}{\bar{u}(y)}\big)}}{|x-y|^{n+2s}} ~ dx dy\leq \frac{C_{n,s}}{2}\int_{D_{\Omega}}\frac{|v_a(x)-v_a(y)|^2}{|x-y|^{n+2s}}\,dx\,dy\leq  [v_a]^2_s
  \end{equation}
  Adding equation \eqref{ee2.22} and \eqref{ee2.21}, it follows from \eqref{e2.19}, we obtain
  \begin{equation}\label{e2.24}
  \begin{split}
      \int_{\Pi_2}\frac{|v_a|^2}{\bar{u}} \mathcal{N}_s \bar{u} ~dx + \int_{\overline{\Pi_2}\cap\partial\Omega} \frac{|v_a|^2}{\bar{u}}\frac{\partial \bar{u}}{\partial \nu}\,d{\sigma} +\int_{\Omega} \mathcal{L}\bar{u}\frac{|v_a|^2}{\bar{u}}  ~dx
      &\leq \|\nabla v_a\|^2_{L^2({\Omega})}+[v_a]^2_s=\eta(v_a)^2.
    \end{split}
    \end{equation}
   Hence, \eqref{eeA} follows from \eqref{e2.24} by passing on the limit $\delta \to 0$ and $a\to +\infty$ while using Fatou's Lemma together with the monotone convergence {T}heorem.
 \end{proof}
 As a result, we have the following  Brezis and Kamin type comparison principle in the mixed operator.
\begin{lemma}\label{l2.13}
 If $g(x, \varrho)$ is a Carath\'eodory function such that $\frac{g(x,\varrho)}{\varrho}$ is decreasing for $\varrho{>0},$ uniformly with respect to a.e. $x\in \R^n$ and $u, v \in  {\mathcal{X}^{1,2}(D)}$ satisfies
 \begin{equation}\label{ee3.17}
    \left\{\begin{matrix} \mathcal{L}v\: \geq g(x,v), ~~~
     v>0 ~in ~ \Omega,~~~\\
     \mathcal{L}u\: \leq g(x,u)
     ,~~~u>0 ~in ~ \Omega. ~~~\\
    \end{matrix} \right.
\end{equation}
Then $v\geq u$ a.e. in $\R^n$.
\end{lemma}
{\begin{proof} We shall only give some ideas of the proof here, for details one can refer to \cite{amundsen2023mixed}.
Firstly we claim that $|K|=0$, where 
{$$K=\{x\in \R^n: u(x)>v(x)\}.$$}
Let us define $\bar{u}=u. 1_{K}$, $\bar{v}=v. 1_\textbf{K}$, where $1_\textbf{K}$ is a characteristic function over $K$ and $w=\bar{u}^2-\bar{v}^2$.  then due to Theorem \ref{T2.12} ( Picone inequality), we deduce that $\frac{w}{\bar{u}}, \frac{w}{\bar v}$ can be treated as test functions for solutions to \eqref{ee3.17}. Hence, we can write  
\begin{equation}\label{ee3.18}
    \int_{\Omega} \bar u \mathcal{L}\left(\frac{w}{\bar u}\right)\,dx-\int_{\Omega}\bar v \mathcal{L}\left(\frac{w}{\bar v}\right)\,dx\leq \int_{\Omega} \left(\frac{g(x,\bar u)}{\bar u}-\frac{g(x,\bar v)}{\bar v}\right)w\,dx.
\end{equation}
By assumption, $\frac{g(x,\varrho)}{\varrho}$ is decreasing which says that the right-hand side of equation \eqref{ee3.18} is negative $K$. Our proof is complete if we prove that the left-hand side in \eqref{ee3.18} is positive. In fact, if this is true then $w\equiv 0$ and so $u\leq v$ a.e. in $\R^n$. Let us now demonstrate that the {left}-hand side in equation \eqref{ee3.18} is non-negative. 
It will then be enough to show that
\begin{equation}\label{ee3.19}
\begin{split}
    &\nabla v(x).\nabla \left( \frac{\bar{u}(x)^2-\bar{v}(x)^2}{v(x)}\right)+(v(x)-v(y)) \left( \frac{\bar{u}(x)^2-\bar{v}(x)^2}{v(x)}-  \frac{\bar{u}(y)^2-\bar{v}(y)^2}{v(y)}\right)\\
&\leq 
\nabla u(x).\nabla \left( \frac{\bar{u}(x)^2-\bar{v}(x)^2}{u(x)}\right)+(u(x)-u(y)) \left( \frac{\bar{u}(x)^2-\bar{v}(x)^2}{u(x)}-  \frac{\bar{u}(y)^2-\bar{v}(y)^2}{u(y)}\right).
\end{split}   
\end{equation}
We closely follow the proof of {L}emma 3.11 in \cite{amundsen2023mixed}] and use the same notation as there to claim that
\begin{equation}\label{e3.18}
    Q_1+Q_2\leq Q_3+Q_4
    \end{equation} 
where,
\begin{align*}
    Q_1&= \nabla v(x).\nabla\left(\frac{\bar{u}(x)^2}{v(x)}\right)+(v(x)-v(y))\left(\frac{\bar{u}(x)^2}{v(x)}-\frac{\bar{u}(y)^2}{v(y)}\right),\\
    Q_2&= \nabla u(x).\nabla\left(\frac{\bar{v}(x)^2}{u(x)}\right)+(u(x)-u(y))\left(\frac{\bar{v}(x)^2}{u(x)}-\frac{\bar{v}(y)^2}{u(y)}\right),\\
    Q_3&=\nabla v(x).\nabla \bar{v}(x)+(v(x)-v(y))(\bar{v}(x)-\bar{v}(y)),\\
    Q_4&=\nabla u(x).\nabla \bar{u}(x)+(u(x)-u(y))(\bar{u}(x)-\bar{u}(y)).
\end{align*}
Let us now identify the following cases:\\
{\bf{Case 1}}:  If $x,y \in K$, by using Picone-type inequality i.e. Theorem \ref{T2.12}, we have $Q_1\leq Q_3$ and $Q_2\leq Q_4.$
Hence \eqref{e3.18} will be satisfied.
\vspace{0.1cm}\\
{\bf{Case 2}}:  If $x,y \notin K$,
{one can verify that} \eqref{e3.18}  {holds due to}  definition of $\bar{u}(x)$ and $\bar{v}(x)$. {Since for}  this case, $Q_i=0$ for all $i=1,2,3,4.$
\vspace{0.1cm}\\
{{\bf{Case 3}}:  If $x \in K, y \notin K$, we obtain 
\begin{align*}
Q_1+Q_2
=\nabla v(x).\nabla\left(\frac{\bar{u}(x)^2}{v(x)}\right)+\bar{u}(x)^2-\bar{u}(x)^2\frac{{v}(y)}{v(x)}+\nabla u(x).\nabla\left(\frac{\bar{v}(x)^2}{u(x)}\right)+\bar{v}(x)^2-\bar{v}(x)^2\frac{{u}(y)}{u(x)}
\end{align*}
and also
\begin{align*}
    Q_3+Q_4= \nabla v(x).\nabla \bar{v}(x)+\bar{v}(x)^2-\bar{v}(x)v(y)+ \nabla u(x).\nabla \bar{u}(x)+\bar{u}(x)^2-\bar{u}(x)u(y).
    \end{align*}
Now, by using relation \eqref{e3.18} and classical Picone inequality \cite{MR1618334}, we have
$$
u(y)\left(\bar{u}(x)-\frac{\bar{v}(x)^2}{u(x)}\right)\leq v(y) \left(\frac{\bar{u}(x)^2}{v(x)}-\bar{v}(x)\right).$$ 
It is also equivalent to 
$u(y)\bar{v}(x)\leq v(y)\bar{u}(x).$
Since, $x\in K$ and $y\notin K$, the above inequality remains true.}
\vspace{0.1cm}\\
{\bf{Case 4}}:  If $x \notin K, y \in K$, then
this case is simply done by interchanging the roles of variable $x$ and $y$ in above case 3.
Finally, we have seen that the inequality \eqref{ee3.19} is always true, as desired and
as a result, the proof is complete.
\end{proof}
}
Finally, we present the last result of this section which is a compactness {L}emma.
\begin{lemma} \label{l2.14}
 Suppose $\{u_n\}$ be a bounded  sequence of non-negative functions in ${\mathcal{X}^{1,2}(D)}$ such that  $u_n \rightharpoonup u $ weakly in ${\mathcal{X}^{1,2}(D)}$ as $n \to \infty$ and $u_n \leq u$  a.e. in $\R^n$, for some $u\in {\mathcal{X}^{1,2}(D)}$. If $\mathcal{L}u_n\geq 0$ in $\Omega$, for all $n$, then $u_n \to u$ strongly in ${\mathcal{X}^{1,2}(D)}$ as $n \to \infty$.
\end{lemma}
\begin{proof}
   Since $u_n\leq u$ a.e. in $\R^n$ and $\mathcal{L}u_n\geq 0$ in $\Omega$ by assumption, we obtain
$$\int_{\Omega}(u-u_n) \mathcal{L}u_n dx \geq 0 ~\implies~\int_{\Omega}u\mathcal{L} u_n  dx \geq \int_{\Omega}u_n\mathcal{L} u_n ~ dx.$$
   Using Young's inequality above, we have
   $$ \int_{\Omega}|\nabla u_n|^2 dx+ 
    \frac{C_{n,s}}{2}\int_{D_{\Omega}} {\dfrac{(u_n(x)-u_n(y))^2}{|x-y|^{n+2s}}} ~ dx dy \leq \int_{\Omega}|\nabla u|^2 dx 
   +\frac{C_{n,s}}{2} \int_{D_{\Omega}} {\dfrac{(u(x)-u(y))^2}{|x-y|^{n+2s}}} ~ dx dy.$$
   Thus, $$ \eta(u_n)^2\leq \eta(u)^2 ~\implies~\limsup_{n \to \infty} \eta(u_n)^2\leq \eta(u)^2. $$
   Now due to {P}roposition \ref{p2.2} and $u_n\rightharpoonup u$ weakly in ${\mathcal{X}^{1,2}(D)}$, the above inequality suggests
   \begin{align*}
        \limsup_{n \to \infty} \eta(u_n-u)^2
       &= \limsup_{n \to \infty} ~(\eta(u_n)^2+ \eta(u)^2-2 \langle u_n, u\rangle )\\
       &\leq 2\eta(u)^2-2 \limsup_{n \to \infty}\langle u_n, u\rangle = 2\eta(u)^2 - 2\eta(u)^2=0
   \end{align*}
       Hence $\limsup_{n \to \infty} \eta(u_n-u)^2= 0$
     which says that $u_n \to u$ strongly in ${\mathcal{X}^{1,2}(D)}$ as $n\to\infty.$
   \end{proof}
 \section{Proof of main results}
   This section contains the proof of Theorem \ref{T1.1} and Theorem \ref{T2.1} presented in two subsections. We first notice that 
 \begin{align*}
     J_{\lambda}(u)
     \geq \frac{1}{2}\eta(u)^2 -{\lambda}C_1 \eta(u)^{\frac{q+1}{2}}- C_2 \eta(u)^{\frac{p+1}{2}}
     \end{align*}
for some constants $C_1, C_2>0$, due to Remark \ref{l1}.
Based on the geometry of the function 
$ f(t) = \frac{1}{2} t^2 - \lambda C_1 {t}^{\frac{q+1}{2}}
 - C_2 {t}^{\frac{p+1}{2}}$, the existence of two solutions of \eqref{1} can be established as- the first solution through minimization technique while the second through the mountain pass {T}heorem. This is a standard technique and one can refer \cite{MR0370183}, for details.

   \subsection{Proof of Theorem \ref{T1.1}}
 We split the proof into various Lemmas and conclude the final result at the end of this subsection.  We recall that {$J_{\lambda}$ is well defined and  differentiable on ${\mathcal{X}^{1,2}(D)}$ and for any $\varphi\in {\mathcal{X}^{1,2}(D)}$}, 
$$J'_{\lambda}(u)(\varphi):= \int_{\Omega} \nabla u\cdot\nabla \varphi~ dx+ \frac{C_{n,s}}{2}\int_{D_{\Omega}} \frac{(u(x)-u(y))(\varphi(x)-\varphi(y))}{|x-y|^{n+2s}} \, dx\,dy -{\lambda}\int_{\Omega}u^{q}\varphi \,dx - \int_{\Omega}u^{p}\varphi \,dx\\.$$
Hence, critical points of the functional $J_{\lambda}$ are solutions to the problem \eqref{1}.
\begin{lemma} \label{le3.1}
Let 
$ \Lambda = \sup\{ \lambda>0 : \text{Problem \eqref{1} has a solution}\}~ \text{then} ~0<\Lambda<\infty.$
\end{lemma}
\begin{proof}
We first prove that $\Lambda>0$ which is equivalent to showing $(P_\lambda)$  has a solution for some $\lambda>0$. As discussed in Section $3$, taking $g\equiv 1$, we know that \eqref{2.6} has a solution $W\in \mathcal{X}^{1,2}(D)$. Moreover since $g\equiv 1\in L^\infty (D)$, we can apply Lemma \ref{l2.8} to conclude that $W\in L^\infty(D)\subset L^\infty(\Omega)$. Therefore we can find a $\lambda_0>0$ such that for all $\lambda \in (0,\lambda_0]$, there exists a $M=M(\lambda)>0$ such that 
\[M\geq \lambda M^q\|W\|_{L^\infty(\Omega)}^q + M^p\|W\|_{L^\infty(\Omega)}^p.\]
Hence, $MW$ forms a supersolution of \eqref{1}. Now we recall $0<\phi_1 \in L^\infty(\Omega)$ as eigenfunction corresponding to first eigenvalue $\lambda_1(D)$ of $\mathcal{L}$ with mixed boundary, proved in \cite{}. For sufficiently small $\epsilon>0$, one can show that $\epsilon \phi_1$ is a subsolution of \eqref{1} such that $MW>\epsilon \phi_1$ in $D$. Hence we apply Lemma \ref{l2.7} to obtain a solution $u \in \mathcal{X}^{1,2}(D)$ of $(P_\lambda)$ such that $\epsilon \phi_1 \leq u\leq MW$. This proves that $\Lambda >0$.

    We now prove that $\Lambda<\infty$, so suppose $\lambda\in \Lambda$ and $\bar{u}$ is a solution to $\eqref{1}$. We consider the following problem
    \begin{equation} \label{eq3.1}
    \left\{\begin{split} \mathcal{L}u\: &= u^q,~~u>0~ \text{in} ~\Omega, \\
      u&=0~~\text{in} ~~{D^c},\\
 \mathcal{N}_s(u)&=0 ~~\text{in} ~~{\Pi_2}, \\
 \frac{\partial u}{\partial \nu}&=0 ~~\text{in}~~ \partial \Omega \cap \overline{\Pi_2}.
    \end{split} \right.
\end{equation}
    where $q\in (0,1)$.
 To find a solution to problem \eqref{eq3.1}, we consider the following minimization problem 
 $$
 \mathcal{M}= \min\limits_{w\in {\mathcal{X}^{1,2}(D)}}\left\{ \frac{1}{2}\eta(w)^2-\frac{1}{q+1}\int_{\Omega}(w^+)^{q+1}\,dx \right\}.
 $$
{It is easy to see that $\mathcal{M}$ is achieved by a minimizer $Z\in \mathcal{X}^{1,2}(D)$. Let us write $Z=Z^+-Z^-$, where $Z^\pm=\max\{\pm Z,0\}$. Then by using $Z$ is a solution to problem \eqref{eq3.1} and testing with $-Z^-$ together with the following inequality
$$|Z^-(x)-Z^-(y)|^2\leq (Z(x)-Z(y))( -Z^-(x)+Z^-(y))~\text{for a.e.}~x,y\in \mathbb{R}^n,$$
one can easily obtain $Z\geq 0$ a.e. in $\mathbb{R}^n$. Now we can use Corollary \ref{p2.11} and Lemma \ref{l2.13} to obtain $Z > 0$ a.e. in $\mathbb{R}^n$ and it is unique}. In addition, by using Theorem \ref{t2.10}, we have $Z\in L^{\infty}({D})$. Now
we set $\bar{Z}_{\lambda} = \lambda^{\frac{1}{1-q}}Z$ for $\lambda>0$ and find that $\bar{Z}_{\lambda}$ solves
\begin{equation} \label{eq4.2}
    \left\{\begin{split} \mathcal{L}\bar{Z}_{\lambda}\: &= \lambda {\bar{Z}_{\lambda}}^q,~~ \bar{Z_{\lambda}}>0~\text{in} ~\Omega, \\
      \bar{Z_{\lambda}}&=0~~\text{in} ~~{D^c},\\
 \mathcal{N}_s(\bar{Z_{\lambda}})&=0 ~~\text{in} ~~{\Pi_2}, \\
 \frac{\partial \bar{Z_{\lambda}}}{\partial \nu}&=0 ~~\text{in}~~ \partial \Omega \cap \overline{\Pi_2}.
    \end{split} \right.
\end{equation}
  By {L}emma \ref{l2.13}, we get $\bar{Z}_{\lambda}\leq \bar{u}$ a.e. in $\R^n$. Now suppose $\Phi \in {\mathcal{X}^{1,2}(D)}$, then  using Theorem \ref{T2.12} and the fact that $\bar{u}$ solves \eqref{1}, we can write
\begin{align*}
     \int_{\Omega}|\nabla \Phi|^2~ dx+\frac{C_{n,s}}{2}\int_{D_{\Omega}} {\dfrac{(\Phi(x)-\Phi(y))^2}{|x-y|^{n+2s}}} ~ dx dy &\geq \int_{\overline{\Pi_2}\cap\partial\Omega}\frac{\Phi^2}{\bar{u}} \frac{\partial \bar{u}}{\partial \nu}\,d{\sigma}  + \int_{\Pi_2}\frac{\Phi^2}{\bar{u}} \mathcal{N}_s \bar{u} \,dx +\int_{\Omega} \frac{\Phi^2}{\bar{u}} \mathcal{L}\bar{u}~ dx\\
    &{= \int_{\Omega} {\Phi}^2 (\lambda {\bar{u}}^{q-1}+ {\bar{u}^{p-1}})~ dx\geq \int_{\Omega} {{\bar{Z}_{\lambda}}}^{p-1} \Phi^2 ~dx}= \lambda^{\frac{p-1}{1-q}}\int_{\Omega} {{Z}}^{p-1}\Phi^2~ dx.
 \end{align*}
 Since $\Phi$ was arbitrary, we conclude that,
 $$ \lambda^{\frac{p-1}{1-q}}\leq \inf_{\Phi\in {\mathcal{X}^{1,2}(D)}}\frac{ \int_{\Omega}|\nabla \Phi|^2~ dx+\frac{C_{n,s}}{2}\int_{D_{\Omega}} {\dfrac{(\Phi(x)-\Phi(y))^2}{|x-y|^{n+2s}}} ~ dx dy }{{\int_{\Omega} Z^{p-1}\Phi^2~ dx}}= \Lambda^*.$$
 Hence we observe that  ${0<\lambda< \Lambda^*}^{\frac{1-q}{p-1}}<\infty$ which finishes the proof.
 \end{proof}
\begin{lemma}\label{ll4.2}
Let 
\begin{equation}\label{eq3}
 \mathcal{S} = \{\lambda > 0 :  \text{Problem \eqref{1} has a solution\}},
 \end{equation}
then $\mathcal S$ is an interval.
\end{lemma}
\begin{proof}
 Let $\mu_1\in S$ be fixed, we have to prove that for all $0<\mu_2<\mu_1$, problem $(P_{\mu_2})$ has a nontrivial solution. As we know $\mu_1\in S$, then we get the existence of $v_{\mu_1}\in \mathcal{X}^{1,2}(D) $ such that $v_{\mu_1}$ solves $(P_{\mu_1})$. Clearly, $v_{\mu_1}$ is a super-solution to problem $(P_{\mu_2})$. Recalling, $Z$ is the unique solution to problem \eqref{eq3.1}, setting $\zeta= \mu_2^\frac{1}{1-q}Z$, we can find that $\zeta$ solves the following problem
\begin{equation} 
    \left\{\begin{split} \mathcal{L}\zeta\: &= \mu_2 \zeta^q,~~ \zeta>0~\text{in} ~\Omega, \\
      \zeta&=0~~\text{in} ~~{D^c},\\
 \mathcal{N}_s(\zeta)&=0 ~~\text{in} ~~{\Pi_2}, \\
 \frac{\partial \zeta}{\partial \nu}&=0 ~~\text{in}~~ \partial \Omega \cap \overline{\Pi_2}.
    \end{split} \right.
\end{equation}
From Lemma \eqref{l2.13}, it holds $\zeta\leq v_{\mu_1}$ a.e. in $\mathbb R^n$. Since $\zeta$ is a subsolution to problem $(P_{\mu_2})$, then using a monotonicity argument we get the existence of solution $u\in \mathcal{X}^{1,2}(D)$ such that $\zeta\leq u\leq v_{\mu_1}$ a.e. in $\R^n$, and $u$ solves $(P_{\mu_2})$. Hence, $\mu_2\in S$ and the result holds.
\end{proof}

Let us look at the energy properties of positive solutions of \eqref{1}.
\begin{lemma}\label{ll3.2}
If problem $(P_{\lambda})$ has a positive solution  for  $0<\lambda<\Lambda,$ then it has a minimal solution $u_{\lambda}$ such that $J(u_{\lambda})<0$, for all $0<\lambda<\Lambda.$ 
Furthermore, the family of minimal solutions $u_{\lambda}$ is increasing with respect to $\lambda.$
\end{lemma}
\begin{proof}
Let us reconsider 
$\bar{Z}_{\lambda}$ which solves \eqref{eq4.2}. We already know that for each $\lambda \in (0, \Lambda)$, there exists a positive solution $w_\lambda$ of \eqref{1}, by definition of $\Lambda$. Since $\mathcal{L}\bar{Z}_{\lambda}= {\lambda{\bar{Z}_{\lambda}}^q}$ in $\Omega$ and $\mathcal{L}w_\lambda \geq \lambda w_\lambda^q$ in $\Omega$, so we use {L}emma \ref{l2.13} with $v= w_\lambda$ to obtain that  $w_\lambda\geq \bar{Z_{\lambda}}$ a.e. in $\R^n$.
Thus, ${\bar{Z}_{\lambda}}$ is a subsolution of \eqref{1}. We set  $u_0=\bar{Z}_{\lambda}$ and consider the monotone iteration, where $u_n$ solves the following problem
\begin{equation} \label{e3.5}
    \left\{\begin{split} \mathcal{L}u_n\: &= \lambda u^q_{n-1}+ u^p_{n-1}~~ {u_n}>0~\text{in} ~\Omega, \\
      u_n&=0~~\text{in} ~~{D^c},\\
 \mathcal{N}_s(u_n)&=0 ~~\text{in} ~~{\Pi_2}, \\
 \frac{\partial u_n}{\partial \nu}&=0 ~~\text{in}~~ \partial \Omega \cap \overline{\Pi_2}.
    \end{split} \right.
\end{equation}
{Then $\{u_n\}$ is an increasing sequence and satisfies 
 $\bar{Z}_{\lambda}\leq\dots\leq u_{n-1}\leq u_n\leq w_{\lambda}$ a.e. in $\R^n$, for each $n$. So, by  Corollary \ref{p2.11}, it follows that $\bar{Z}_{\lambda}<u_n<w_{\lambda}$ a.e. in $\R^n$, for each $n$.
Due to this, taking $u_n$ as a test function in \eqref{e3.5}, one can obtain $\eta(u_n)\leq \eta(w_{\lambda})$.}
Hence, there {exists} some $u_{\lambda}\in {\mathcal{X}^{1,2}(D)}$ such that $u_n\rightharpoonup u_{\lambda}$ weakly in ${\mathcal{X}^{1,2}(D)}$ and pointwise a.e. as $n \to \infty.$ {So we can use {L}emma \ref{l2.14}, since $\mathcal{L} u_n\geq 0$ in $\Omega$ and $u_n\leq u_{\lambda}$ a.e. in $\R^n$, then we obtained that sequence $u_n \to u_{\lambda}$ strongly converges in ${\mathcal{X}^{1,2}(D)}$ and also $u_{\lambda}\leq w_{\lambda}$} a.e. in $\R^n$. Since $w_\lambda$ is any weak solution of \eqref{1}, $u_{\lambda}$ must be the minimal solution of \eqref{1}.
Finally, applying {C}orollary \ref{p2.11} and {L}emma \ref{l2.13}, the monotonicity of the family $\{u_{\lambda}, \lambda\in (0,\Lambda)\}$ follows easily. We will later demonstrate, how this minimal solution $u_{\lambda}$  satisfies $J(u_{\lambda})<0$, from following Lemma \ref{l4.4} and Lemma \ref{l4.5}. 
\end{proof}
\begin{lemma}\label{l4.4}
Let $\bar{Z_{\lambda}}$ and $ w_{\lambda}$ are subsolution and supersolution of \eqref{1}, respectively such that $\bar{Z_{\lambda}}< w_{\lambda}$ a.e. in $\R^n$. {Also let} $u_{\lambda}$ be a minimal solution of \eqref{1} such that 
 $\bar{Z_{\lambda}}\leq u_{\lambda}\leq w_{\lambda}$ a.e. in $\R^n$. Then ${\lambda^*}=\eta_1\geq 0$, where $\eta_1$ is {the} first eigenvalue of $(\mathcal{L}-b(x))$ with mixed boundary data as in $(P_\lambda)$ and $b(x) = \lambda q u^{q-1}_{\lambda}+ p u^{p-1}_{\lambda}.$ 
 \end{lemma}
\begin{proof}
By contrary, we assume ${\lambda^*}<0$ and $\varphi>0$ is a eigenfunction of the following
\begin{equation}\label{new3}
    \left\{\begin{split} (\mathcal{L}-b(x))\varphi \: &= {\lambda^*} \varphi~~ \varphi \geq 0~\text{in} ~\Omega, \\
      \varphi&=0~~\text{in} ~~{D^c},\\
 \mathcal{N}_s(\varphi)&=0 ~~\text{in} ~~{\Pi_2}, \\
 \frac{\partial \varphi}{\partial \nu}&=0 ~~\text{in}~~ \partial \Omega \cap \overline{\Pi_2}.
    \end{split} \right.
\end{equation}
\textbf{Claim}: $u_{\lambda}-a\varphi$ is a supersolution of \eqref{1}, for any sufficiently small $a>0$.\\
By direct calculation, we get
\begin{align*}
\mathcal{L}(u_{\lambda}-a\varphi)-(\lambda(u_{\lambda}-a\varphi)^q+(u_{\lambda}-a\varphi)^p)
&=\lambda u_{\lambda}^q+u_{\lambda}^p-a{\lambda^*}\varphi-a(\lambda q u^{q-1}_{\lambda}+ p u^{p-1}_{\lambda})\varphi\\
&-\lambda(u_{\lambda}-a\varphi)^q-(u_{\lambda}-a\varphi)^p.
\end{align*}
{Since for $0<q<1$, $t\mapsto t^q$ is concave, so we have}
\begin{equation}\label{a}
(u_{\lambda}-a\varphi)^q\leq u_{\lambda}^q-aqu_{\lambda}^{q-1}\varphi.
\end{equation}
Now using above \eqref{a}, we obtain that in $\Omega$
\begin{align*}
\mathcal{L}(u_{\lambda}-a\varphi)-(\lambda(u_{\lambda}-a\varphi)^q+(u_{\lambda}-a\varphi)^p)
&\geq u_{\lambda}^p-a{\lambda^*}\varphi-a p u^{p-1}_{\lambda}\varphi
-(u_{\lambda}-a\varphi)^p \geq-a\lambda^* \varphi +o( a\varphi)>0,
\end{align*}
for $a>0$ small enough, since, ${\lambda^*}<0$ and $\varphi>0$.  
Hence, we can see $u_{\lambda}-a\varphi$ is a super solution of  \eqref{1}.
Furthermore, since $\bar{Z_{\lambda}}$ is not a solution of \eqref{1} so $u_{\lambda}\not \equiv\bar{Z_{\lambda}}$. Therefore for small enough value of $a>0$, one can get $u_{\lambda}-a\varphi\geq \bar{Z_{\lambda}}$ a.e. in $\R^n.$ Applying Lemma \ref{l2.7}, we infer that \eqref{1} has  another solution $u'_{\lambda}$ such that $\bar{Z_{\lambda}}\leq u'_{\lambda}\leq u_{\lambda}-a\varphi$  a.e. in $\R^n$, which is a contradiction to $u_{\lambda}$ being the minimal solution of \eqref{1}.
\end{proof}
\begin{lemma}\label{l4.5}
    If $u_{\lambda}$ is a minimal solution of \eqref{1}, for $0<\lambda<\Lambda$ then  $J_{\lambda}(u_{\lambda})<0.$
\end{lemma}
\begin{proof}
 From Lemma \ref{l4.4}, $\lambda^*\geq 0$. It is easy to see that $\lambda^*\geq 0$ is equivalent to
\begin{equation} \label{e3.7}
\eta(\varphi)^2\geq \int_{\Omega} b(x) {\varphi}^2 dx, ~~\forall~ \varphi\in \mathcal{X}^{1,2}(D). 
\end{equation}
 As we know $u_{\lambda}$ is a solution to \eqref{1} then $J'_{\lambda}(u_{\lambda})=0$ which gives
\begin{equation} \label{e3.8}
  \eta(u_{\lambda})^2= \lambda \|u_{\lambda}\|^{q+1}_{L^{q+1}(\Omega)}+ \|u_{\lambda}\|^{p+1}_{L^{p+1}(\Omega)}.
\end{equation}
From equation \eqref{e3.7}, it follows that
\begin{equation}
 \eta(u_{\lambda})^2- \lambda q~\|u_{\lambda}\|^{q+1}_{L^{q+1}(\Omega)}-p~ \|u_{\lambda}\|^{p+1}_{L^{p+1}(\Omega)}\geq 0.
 \end{equation}
Using the above relations in \eqref{ee1.1}, we can deduce that $J_{\lambda}(u_{\lambda})<0.$ Thus, with the help of this Lemma, proof of Lemma \ref{ll3.2} is now completed.
   \end{proof}
\begin{lemma}\label{proof13}
If $\lambda= \Lambda$ then  \eqref{1} has at least one solution.
\end{lemma}
\begin{proof}
Suppose $\{\lambda_n\}$ be a sequence such that $\lambda_n \nearrow \Lambda$ as $n \to \infty$. For simplicity, we denote $u_n:= u_{\lambda_{n}}$ as the minimal solution of  $(P_{\lambda_n})$ as obtained in {L}emma \ref{ll3.2}, then the sequence $\{u_n\}$ is increasing w.r.t. $\lambda_n$ and from Lemma \ref{ll3.2}, we know that $J_{\lambda_n}(u_n)<0$. Moreover, this says
\begin{align*}
    0> ~& ~J_{\lambda}(u_n)- \frac{1}{p+1}\langle J'_{\lambda}(u_n), u_n\rangle \\
\geq & \left(\frac{1}{2}-\frac{1}{p+1}\right)\eta(u_n)^2+\lambda_n\left(\frac{1}{p+1}-\frac{1}{q+1}\right)\|u_n\|^{q+1}_{L^{q+1}(\Omega)}\\
\geq & \left(\frac{1}{2}-\frac{1}{p+1}\right)\eta(u_n)^2-\lambda_n C\left(\frac{1}{q+1}-\frac{1}{p+1}\right){\eta(u_n)}^{q+1},~~\text{for some constant $C>0$}.
\end{align*}
From this, it follows that  $\{u_n\}$ is bounded in ${\mathcal{X}^{1,2}(D)}$ which implies that there exists a $u_0\in {\mathcal{X}^{1,2}(D)}$ such that up to a subsequence  $u_n \rightharpoonup u_0$ weakly in  ${\mathcal{X}^{1,2}(D)}$, $u_n \to u_0$ strongly in $L^{r}_{loc}(\mathbb{R}^n)$, where $1\leq r<2^{*}$ and hence $u_n \to u_0$ a.e. in $\R^n$ as $n\to\infty$. 
Let $\varphi\in \mathcal{X}^{1,2}(D)$ then due to $u_n\to u_0$ a.e. in $\R^n$,  we have
$\varphi |u_n|^p \to \varphi |u|^p$ a.e. in $\R^n$ and
$|\varphi |u_n|^p|\leq |\varphi||u_n|^p$ a.e. in $\R^n$. As, $u_n \to u_0$ in $L^r_{loc}(\R^n)$, $1\leq r< 2^*$, so there exists measurable function $g\in L^r(\Omega)$ such that $|u_n|\leq g$ a.e. in $\Omega.$ Now, we use the H\"older inequality to obtain the following 
$$ \int_{\Omega} |\varphi| g^p\,dx \leq \left(\int_{\Omega} |\varphi|^{\eta}\,dx \right)^{\frac{1}{\eta}} \left(\int_{\Omega} g^{p \Theta'}\,dx \right)^{\frac{1}{\Theta'}}< +\infty, ~\forall~ \varphi\in \mathcal{X}^{1,2}(D),$$
where $\Theta=\frac{2^* -1}{2^*-1-p}$ and its conjugate $\Theta^{'}= \frac{2^*-1}{p}.$ 
Hence, now using the Lebesgue dominated convergence theorem, we have
$$
\lim_{n \to \infty}\int_{\Omega}\varphi(x) |u_n|^p\,dx= \int_{\Omega}\varphi(x) |u_0|^p\,dx.
$$
Similarly, we can show that
$$
\lim_{n \to \infty}\int_{\Omega}\varphi(x) |u_n|^q\,dx= \int_{\Omega}\varphi(x) |u_0|^q\,dx.
$$
As a consequence, by def{i}nition \eqref{new5}, $u_0$ is a solution to problem \eqref{1} for $\lambda=\Lambda.$
\end{proof}
The proof of Theorem \ref{T1.1}(1) follows from {L}emma \ref{ll3.2} and Theorem \ref{T1.1}(2) follows from the definition of $\Lambda$. Lastly, Theorem \ref{T1.1}(3) is a consequence of {L}emma \ref{proof13}.
    \begin{remark}
    If $p\leq 2^* -1$ then by using Theorem \ref{t2.10}, we can easily prove that any solution of \eqref{1} belongs to ${L^{\infty}}({D}).$
    \end{remark}
\subsection{Proof of Theorem \ref{T2.1}}
     This final section is devoted to establishing the existence of a second positive solution to the problem \eqref{1}. We fix $p<2^*-1$ for this section. As discussed earlier, the critical points of $J_\lambda$ are weak solutions of \eqref{1}. To prove Theorem \ref{T2.1}, we will use the celebrated mountain pass {T}heorem.

\begin{lemma}\label{new}
    For all $\lambda \in (0,\Lambda)$, \eqref{1} has a solution $u$ which is in addition a local minimum of $J_\lambda$ in $C^1$ topology.
\end{lemma}
     \begin{proof}
     The proof follows using the arguments of Lemma 4.1 in \cite{ambro} and the version of Hopf's Lemma proved in \cite{antonini}.
     \end{proof}
 Let us now consider the non-empty set(due to Lemma \ref{new})
 \begin{equation}
     X= \{\lambda>0 : J_{\lambda}~ \text{has a local minimum, say} ~u_{0,\lambda}\}
     \end{equation}
    then clearly, if $\lambda\in X$ then $v=0$ is a local minimum of 
     \begin{equation}
     {J^*_{\lambda}(v)}= \frac{1}{2}\eta(v)^2- \int_{\Omega}F_{\lambda}(x,v) dx,   
     \end{equation}
     where 
     $ F_{\lambda}(x,v)= \int^{v}_{0} f_{\lambda}(x,l)~ dl$
     and
    \begin{equation*}
    f_{\lambda}(x,l)=
      \begin{cases}    
       \lambda ((u_{0,\lambda}(x)+l)^{q}-u_{0,\lambda}(x)^q) + (u_{0,\lambda}(x)+l)^{p}-u_{0,\lambda}(x)^p, ~\text{if}~ l\geq 0 ,\quad\\ 
      0,~\text{if}~ l<0.
  \end{cases}
 \end{equation*} 
 Easily, we can see that functional $J^*_{\lambda}$ satisfies the mountain pass geometrical structures. Thus, there exists $v_0\in {\mathcal{X}^{1,2}(D)}$ such that $J^*_{\lambda}(v_0)<0$ and we define the mountain pass level as follows $$c= {\inf_{{\Gamma\in\Upsilon}}\max_{t\in [0,1]}}(J^*_{\lambda}(\Gamma(t))),$$
 where
 $$ \Upsilon=\{ \Gamma \in {C( [0,1], \mathcal{X}^{1,2}({D}))},~ \Gamma(0)=0, \Gamma(1)= v_0\}.$$
 It is well known that $c>0$ and since $p< 2^*-1$, then the functional $J^*_{\lambda}$ satisfies the $PS_c$ conditions (Palais- Smale conditions) and using the Ambrosetti-Rabinowitz {T}heorem, we find  a nontrivial critical point, if $c > 0$. The Ghoussoub-Preiss Theorem is used if $c = 0$, as shown in \cite{MR1030853}.
 As a result, if we start with a local minimum of the functional $J^*_{\lambda}$, we find a second  critical point of the functional $J^*_{\lambda}$ and consequently, a second solution to problem \eqref{1}.
 In order to demonstrate that problem \eqref{1} possesses a second solution for all $\lambda\in (0,\Lambda)$, we employ arguments similar to those presented by Alama in \cite{MR1729392}, taking into account the mixed local and nonlocal characteristics of the operator.
 By utilizing a variational formulation of Perron's method, we establish the presence of a restricted minimum for the functional, followed by demonstrating that this minimum constitutes a local minimum across the entire ${\mathcal{X}^{1,2}(D)}$ space. In order to achieve this, we use a truncation approach, accompanied by energy estimates.

 \vspace{0.5cm}
 We fix $\lambda'\in (0,\Lambda)$ such that $\lambda' < \bar{\lambda} <  \Lambda$. Let us define $u_0, \bar{u}$  to be the minimal solutions
to problem \eqref{1} with $\lambda = \lambda'$ and $\lambda = \bar{\lambda} $ respectively. From {L}emma \ref{ll3.2}, we have
$u_0 < \bar{u}$~{a.e.~ in ~$\R^n$}.\\
Let us define
$$ \mathcal{N}=\{u\in {\mathcal{X}^{1,2}(D)}: ~ 0\leq u(x)\leq \bar{u}(x) ~ a.e.~ in ~{\R^n}\}$$
It is clear that $u_0\in \mathcal{N}$ and that $\mathcal{N}$ is a convex closed subset of ${\mathcal{X}^{1,2}(D)}$. Since, $J_{\lambda'}$
is bounded from below in $\mathcal{N}$ and also lower semi-continuous, then we get the existence of
$\mathcal{V} \in \mathcal{N}$ such that
$$
J_{\lambda'}(\mathcal{V})=\inf_{u\in \mathcal{N}} J_{\lambda'}(u).$$
Suppose $v$ is the unique solution to
problem
\begin{equation}\label{e413}
    \left\{\begin{split} \mathcal{L}u\: &= \lambda' u^q ~~u>0~ \text{in} ~\Omega, \\
      u&=0~~\text{in} ~~{D^c},\\
 \mathcal{N}_s(u)&=0 ~~\text{in} ~~{\Pi_2}, \\
 \frac{\partial u}{\partial \nu}&=0 ~~\text{in}~~ \partial \Omega \cap \overline{\Pi_2}.
    \end{split} \right.
\end{equation}
We have that $J_{\lambda'}(v)<0$. Since $\bar{u}$ is a super solution of \eqref{e413} and from $0\leq v\leq \bar{u}$ a.e. in $\R^n$ then we have $J_{\lambda'}(\mathcal{V})\leq J_{\lambda'}(v)<0$ and $v\in \mathcal{N}$. Thus $\mathcal{V}\neq 0.$ 
Repeating the proof of Theorem 2.4 in \cite{1994michael}, {we can infer that $\mathcal{V}$ is a non trivial solution for problem $(P_{\lambda'})$.}
If $\mathcal{V}\neq u_{0}$, then the proof of Theorem \ref{T2.1} is complete. On the other hand, if $\mathcal{V} = u_0$ then we claim that $\mathcal{V}$ is a local minimum of $J_{\lambda'}$. By contrary, 
we are assuming that there exists a sequence
$u_n \in {\mathcal{X}^{1,2}(D)}$ with
\begin{equation}\label{eeq5.4}
 u_n \to \mathcal{V}~ \text{strongly as}~ n\to \infty ~\text{in}~ {\mathcal{X}^{1,2}(D)} ~\text{and}~ J_{\lambda'}(u_n) < J_{\lambda'}(\mathcal{V}). 
\end{equation} 
Also, let us define
\begin{equation}\label{eqq5.5}
w_n=(u_n-\bar{u})^+=\left\{\begin{matrix} {(u_n-\bar{u})(x),~~\text{if}~~ (u_n-\bar{u})(x)\geq0},\\[2mm]
0,~~~\text{otherwise},
\end{matrix}\right.
\end{equation} and 
\begin{equation}\label{eeqq5.6}
v_n(x)=\max\{0,\min\{u_n,\bar{u}\}\}=
\left\{\begin{matrix} {0, ~~ \text{if}~ u_n(x)\leq 0},\\[2mm]  
~~~~u_n(x), ~~\text{if}~  0\leq u_n(x)\leq \bar{u}(x),\\[2mm]
\bar{u}(x), ~~\text{if}~ u_n(x)\geq \bar{u}(x),
\end{matrix}\right. \end{equation}
and easy to see $v_n\in \mathcal{N}$.
{Thus}  $v_n=(u_n^- + w_n)$, where $u_n^- =\min\{0,u_n\}.$ {Let $H_n=\{x\in \Omega: v_n(x)=u_n(x)\}$ and   
   ${S_n}=(\text{supp}~ w_n) \cap \Omega$}.  
  We identify that $(\text{supp}~ u_n^+)\cap\Omega= H_n \cup S_n$.
 Our claim is that 
 \begin{equation}\label{eq5.6}
     |{S_n}|\to 0 ~\text{as}~ n \to \infty.
 \end{equation}
For $\delta>0$, we define
$$A_n=\{x\in \Omega: u_n(x)\geq\bar{u}(x)>\mathcal{V}(x)+\delta\}~$$ 
and
$$B_n=\{x\in \Omega: u_n(x)\geq\bar{u}(x) ~\text{and}~ \bar{u}(x)\leq \mathcal{V}(x)+\delta\}$$
Since
\begin{align*} 
 0= \left|\{x\in \Omega:\bar{u}(x)<\mathcal{V}(x)\}\right|= {\bigcap}^{\infty}_{k=1}\bigg|\bigg\{{x\in \Omega:\bar{u}(x)<\mathcal{V}(x)+\frac{1}{k}\bigg\}}\bigg|=\lim_{k\to \infty}\bigg|\bigg\{{x\in \Omega:\bar{u}(x)<\mathcal{V}(x)+\frac{1}{k}\bigg\}}\bigg|,
\end{align*}
where we have used Theorem 1.8 of \cite{folland1999real}.
For $\epsilon>0$, there exists $k_0\geq 1$ sufficiently large such that choosing $\delta<\frac{1}{k_0}$, it must be $|\{x\in \Omega: \bar{u}(x)<  \mathcal{V}(x)+\delta\}|\leq\frac{\epsilon}{2}$, which implies that $|B_n|\leq\frac{\epsilon}{2}$, for all $n$. 
On the other hand, from  \eqref{eeq5.4}, we have $u_n\to \mathcal{V}$ in $\mathcal{X}^{1,2}(D)$ as $n\to \infty.$
which gives $\|u_n-\mathcal{V}\|_{L^2(\Omega)} \to 0$
as $n \to \infty$, due to Remark \ref{l1}. For $n\geq n_0$ large enough, we obtain that 
$$ \frac{\epsilon \delta^2}{2}\geq \int_{\Omega} |u_n-\mathcal{V}|^2\,dx\geq \int_{A_n} {\delta}^2\,dx\geq \delta^2 |A_n|,$$
which implies that $|A_n|\leq \frac{\epsilon}{2}$, for any $n\geq n_0$. 
Now, we observe that ${S_n}\subset A_n\cup B_n$ which says  $|{S_n}|\leq |A_n|+|B_n|\leq \frac{\epsilon}{2}+\frac{\epsilon}{2}=\epsilon$, for $n\geq n_0$.
Hence, our claim \eqref{eq5.6} is done.
Let us define  
$$
P(u)=\frac{\lambda' {(u^{+})}^{q+1}}{q+1}+\frac{{(u^{+})^{p+1}}}{p+1}.$$
We can easily see the following inequalities: 
\begin{equation}\label{ea5.8}
    \eta(u_n)^2 \geq {\eta(u_n^+)^2+\eta(u_n^-)^2.}
\end{equation}
 Using \eqref{ea5.8} and $w_n+\bar{u}=u_n$ {on $S_n$, we have
\begin{align*}
J_{\lambda'}(u_n)
&=\frac{1}{2}\eta(u_n)^2-\int_{\Omega} P(u_n)\,dx\geq \frac{1}{2}\eta(u_n^+)^2+\frac{1}{2}\eta(u_n^-)^2-\int_{\Omega} P(u_n)\,dx\\
&= \frac{1}{2}\eta(u_n^+)^2+\frac{1}{2}\eta(u_n^-)^2-\int_{{H_n}} P(v_n)\,dx- \int_{{S_n}} P(u_n)\,dx,~\text{since}~ u_n=v_n~ \text{in}~  H_n\\
&= \frac{1}{2}\eta(u_n^+)^2+\frac{1}{2}\eta(u_n^-)^2-\int_{{H_n}} P(v_n)\,dx-\int_{{S_n}} P(w_n+ \bar{u})\,dx\\
&=\frac{1}{2}\eta(u_n^+)^2+\frac{1}{2}\eta(u_n^-)^2+\int_{\Omega} P(v_n)\,dx+\int_{S_n}P(\bar{u})\,dx-\int_{{S_n}} P(w_n+ \bar{u})\,dx\\
&=J_{\lambda'}(v_n)+\frac{1}{2}\eta(u_n^+)^2+\frac{1}{2}\eta(u_n^-)^2-\frac{1}{2}\eta(v_n)^2-\int_{{S_n}} \left(P(w_n+\bar{u})-P(\bar{u})\right)\,dx,
\end{align*}
since $v_n=\bar{u}$ in $S_n.$ It is easy to see that  $u_n^+= v_n+w_n,$ then
\begin{equation}\label{w1}
    \eta(u_n^+)^2= \eta(v_n+w_n)^2
\end{equation}
and  we deduce that
$$\frac{1}{2}\left(\eta(u_n^+)^2-\eta(v_n)^2\right)=\frac{1}{2}\eta(w_n)^2+\langle v_n, w_n\rangle $$
where,
$$\langle{ v_n},{ w_n}\rangle := \int_{{\Omega}} \nabla v_n\cdot \nabla{w_n} \,dx + \frac{C_{n,s}}{2} \int_{D_{\Omega}} {\dfrac{(v_n(x)-v_n(y)) (w_n(x)-w_n(y))}{|x-y|^{n+2s}}} ~ dx dy. $$
We can check that 
$\{w_n\neq 0\}=\{v_n=\bar{u}\}$
which says that
\begin{equation}\label{w2}
    {\langle v_n, w_n\rangle\geq \int_{\Omega}\mathcal{L}\bar{u}w_n} \,dx\geq {\lambda'}\int_{S_n}\bar{u}^q w_n\,dx+\int_{S_n}\bar{u}^p w_n\,dx
\end{equation}
Therefore,  recalling that $\Bar{u}$ is a supersolution to problem \eqref{1} for $\lambda = {\lambda'}$ then \eqref{w1} and \eqref{w2} suggests
\begin{align*}
J_{\lambda'}(u_n)
&\geq
J_{\lambda'}(v_n)+ \frac{1}{2}\eta(w_n)^2+\langle v_n, w_n\rangle +\frac{1}{2}\eta(u_n^-)^2-\int_{{S_n}} \left(P(w_n+\bar{u})-P(\bar{u})\right)\,dx\\
&\geq {J_{\lambda'}(\mathcal{V})} + \frac{1}{2} \eta(w_n)^2 + \frac{1}{2}\eta(u_n^-)^2  
- \int_{S_n}  \bigl\{ P(w_n + \Bar{u}) -P(\Bar{u})- \lambda'\Bar{u}^{q}w_n -\Bar{u}^pw_n  \bigr\}  dx
\end{align*}
 with the help of inequality $((a+b)^r-b^r\geq r a b^{r-1}$, where $r\geq 1 $ and $a,b\geq 0$), we obtain that
$$  0 \leq \frac{1}{q+1} (w_n + \Bar{u})^{q+1} -\frac{1}{q+1} \Bar{u}^{q+1}  -\Bar{u}^qw_n \leq  \frac{q}{2}\frac{w_n^2}{\Bar{u}^{1-q}} , $$
and using  Theorem \ref{T2.12}, we find that
\begin{equation}\label{11}
\Bar{\lambda} \int_{\Omega} \frac{w_n^2}{\Bar{u}^{1-q}} dx \leq  \int_{\Omega} \frac{w_n^2}{\Bar{u}}\mathcal{L} \Bar{u} \leq \eta(w_n)^2.
 \end{equation}
Then, by using equation \eqref{11}, we obtain that 
$$ \lambda' \int_\Omega \frac{1}{q+1}(w_n +\Bar{u})^{q+1} -\frac{1}{q+1} \Bar{u}^{q+1}  -\Bar{u}^qw_n \leq \frac{q}{2}\frac{w_n^2}{\Bar{u}^{1-q}} \leq \frac{q}{2} \eta(w_n)^2. $$
We observe that $ p+1\geq 2$, then
$$ 0 \leq \frac{1}{p+1}(w_n +\Bar{u})^{p+1} -\frac{1}{p+1} \Bar{u}^{p+1}  -\Bar{u}^pw_n \leq \frac{p}{2} w_n^2 (w_n +\Bar{u})^{p-1} \leq C(\Bar{u}^{p-1}w_n^2 + w_n^{p+1}) $$
Hence, by using the Sobolev inequality and the fact that $|S_n| \to 0 $ as $n \to \infty$, we have
$$ \int_\Omega \Bigl\{ \frac{1}{p+1}(w_n +\Bar{u})^{p+1} -\frac{1}{p+1} \Bar{u}^{p+1}  -\Bar{u}^pw_n \Bigr\} dx \leq o(1)\eta(w_n)^2.  $$
Hence, we can conclude from the above calculations
\begin{align}
J_{\lambda'}(u_n)
&\geq J_{\lambda'}({\mathcal{V}})
+\frac{1}{2}\eta(w_n)^2(1-q-o(1))+\frac{1}{2}\eta(u_n^-)^2\geq J_{\lambda'}({\mathcal{V}})
+\frac{1}{2}\eta(w_n)^2(1-q-o(1))+o(1).
\end{align}
So, we get that
$$
0>J_{\lambda'}(u_n)-J_{\lambda'}(\mathcal{V})\geq 
\frac{1}{2}\eta(w_n)^2(1-q-o(1))+\frac{1}{2}\eta(u_n^-)^2
$$
Since $q\in (0,1)$, we deduce that $w_n=u_n^-=0$ for sufficiently large value of $n$, recalling $u_n\in \mathcal{N}$ and 
then we get
$$
J_{\lambda'}(u_n)\geq J_{\lambda'}(\mathcal{V}),$$ which is a contradiction with 
$J_{\lambda'}(u_n)<J_{\lambda'}(\mathcal{V})$.

Hence we have $\mathcal{V}$ is a local minimum for functional $J_{\lambda'}$ and $J^*_{\lambda'}$ has $u=0$ as a local minimum and then $J^{*}_{\lambda'}$
 has a nontrivial solution $u'$. As a consequence, $u = \mathcal{V} + u'$ is a solution to problem \eqref{1}, which is different from $\mathcal{V}$. Hence our proof of Theorem \ref{T2.1} is completed now.

 \vspace{0.2cm}

\section{Acknowledgement}
 We are grateful to Prof. Jacques Giacomoni for his valuable suggestions which helped us to complete this article.

\section{Ethics approval and consent to participate}
Not Applicable

\section{Funding}
T. Mukherjee acknowledges the support of the Start-up Research Grant received from DST-SERB, sanction no. SRG/2022/000524.

\end{document}